\documentclass[11pt]{amsart}
\usepackage{amsmath,amssymb,euscript,mathrsfs,pstricks}
\usepackage{hyperref}
\usepackage[all]{xy}
\usepackage{enumerate}

\pushQED{\qed}

\psset{unit=1pt}
\psset{arrowsize=4pt 1}
\psset{linewidth=.5pt}

\newcommand{\define}{\textbf}
\newcommand{\comment}{$\star$ \texttt}

\renewcommand{\phi}{\varphi}

\renewcommand{\tilde}{\widetilde}

\renewcommand{\bar}{\overline}
\renewcommand{\wedge}{\bigwedge}

\newcommand{\C}{\mathbb{C}}

\newcommand{\R}{\mathbb{R}}
\newcommand{\N}{\mathbb{N}}
\newcommand{\Z}{\mathbb{Z}}

\newcommand{\mth}{\mathrm{th}}

\newcommand{\lef}{\left\langle}
\newcommand{\rig}{\right\rangle}

\DeclareMathOperator{\Sym}{Sym}

\DeclareMathOperator{\tr}{tr}

\DeclareMathOperator{\Hom}{Hom}
\DeclareMathOperator{\Spec}{Spec}

\DeclareMathOperator{\pt}{pt}

\DeclareMathOperator{\Ind}{Ind}

\DeclareMathOperator{\Int}{Int}

\DeclareMathOperator{\class}{class}
\DeclareMathOperator{\st}{st}
\DeclareMathOperator{\vol}{vol}

\DeclareMathOperator{\ind}{ind}
\DeclareMathOperator{\BOX}{Box}
\DeclareMathOperator{\Pyr}{Pyr}
\DeclareMathOperator{\GL}{GL}
\DeclareMathOperator{\aff}{aff}
\DeclareMathOperator{\Res}{Res}
\DeclareMathOperator{\conv}{conv}

\newtheorem{theorem}{Theorem}[section]
\newtheorem*{ntheorem}{Theorem}
\newtheorem{lemma}[theorem]{Lemma}

\newtheorem{proposition}[theorem]{Proposition}

\newtheorem{corollary}[theorem]{Corollary}
\newtheorem*{ncor}{Corollary}

\theoremstyle{definition}
\newtheorem{definition}[theorem]{Definition}
\newtheorem{remark}[theorem]{Remark}
\newtheorem{example}[theorem]{Example}

\newtheorem{conjecture}[theorem]{Conjecture}
\newtheorem*{nconj}{Conjecture}

\newcommand{\excise}[1]{}

\begin{document}%%%%%%%%%%%%%%%%%%%%%%
%%%%%%%%%%%%%%%%%%%%%%%%%%%%%%%%%%%%%%

%%%%%%%%%%%%%%%%%%%%%%%%%%%%%%%%%%%%%%%%%%
\title{Equivariant Ehrhart theory}
\author{Alan Stapledon}
\address{Department of Mathematics\\University of British Columbia\\ BC, Canada V6T 1Z2}
\email{astapldn@math.ubc.ca}


\date{March 30, 2010}
%\thanks{The author would like to thank Kalle Karu, Gus Lehrer, 
%John Stembridge and Josephine Yu for helpful comments and enlightening discussions.}

%\subjclass[2000]{55N91, 14E99.}
% 55N91Equivariant homology and cohomology
% 14E99 Birational geometry

\begin{abstract}
Motivated by representation theory and geometry, we introduce and develop an equivariant generalization of Ehrhart theory, the study of lattice points in dilations of lattice polytopes. 
%We prove equivariant analogues of Ehrhart reciproicty
We prove representation-theoretic analogues of numerous classical results, and give
applications to the Ehrhart theory of 
rational polytopes and centrally symmetric polytopes. We also recover  a character formula 
of Procesi, Dolgachev, Lunts and Stembridge for the action of a Weyl group on the cohomology of a toric variety associated to a root system. 
\end{abstract}

\maketitle
%%%%%%%%%%%%%%%%%%%%%%%%%%%%%%%%%%%%%%%%%%%%%%%%%%%
\tableofcontents

%%%%%%%%%%%%%%%%%%%%%%%%%%%%%%%%%%%%%%%%%%%%%%%%%%%
\section{Introduction}%%%%%%%%%%%%%%%%%%%%%%%%%%%%%
%%%%%%%%%%%%%%%%%%%%%%%%%%%%%%%%%%%%%%%%%%%%%%%%%%%

%The goal of this paper is to show that interesting and subtle combinatorics arise when one studies lattice polytopes which are invariant under the action of a finite group. More specifically, 
Let $G$ be a finite group acting linearly on a lattice $M'$ of rank $n$, % via $\rho: G \rightarrow GL(M')$, 
and let $P$ be a $d$-dimensional $G$-invariant lattice polytope.
%Fix a translation $M$  of the intersection of the affine span of $P$ and $M'$ to the origin,
Let $M$ be a translation of the intersection of the affine span of $P$ and $M'$ to the origin, and consider the induced representation $\rho: G \rightarrow GL(M)$ (see Section~\ref{s2}). 
 If $\chi_{mP}$ denotes the permutation character associated to the action of $G$ on the lattice points in the $m^\mth$ dilate of $P$, and $R(G)$ denotes the ring of virtual characters of $G$, then we introduce virtual characters $\{ \phi_i \}_{i \in \N}$ %satisfying %the following equality in $R(G)[[t]]$, 
determined by the equation
\begin{equation*}
\sum_{m \ge 0} \chi_{mP} t^m = \frac{ \phi[t]}
%{ (-1)^{d + 1}(t  - 1)(t^d - V \, t^{d - 1} + \bigwedge^2 V \,  t^{d - 2}  - \cdots + (-1)^d \bigwedge^d V)}
%\bigwedge^d V - \bigwedge^{d - 1} V t + \cdots + (-1)^{d - 1} V t^{d - 1}+ (-1)^d )}
{(1- t)\det[I - \rho t]} \: \textrm{ in } R(G)[[t]], 
\end{equation*}
where $\phi[t] = \phi[P,G; t] = \sum_{i} \phi_i t^i$.  
%and $\det[I - tM](g) = \det(I - tA)$ if $g \in G$ acts on $M$ via an integer matrix $A$. 
These virtual characters naturally appear when one studies the action of a finite group on the cohomology of an invariant hypersurface in a toric variety \cite{YoRepresentations}. 
If we restrict to the action of the trivial group, then $L(m) = \chi_{mP}$ restricts to the \define{Ehrhart polynomial} of $P$, and $\phi[t]$ restricts to the \define{$h^*$-polynomial} of $P$. 
Our goal is  to establish and exploit an equivariant generalization of Ehrhart theory, %which is essentially 
the study of Ehrhart polynomials and $h^*$-polynomials of lattice polytopes (see, for example, \cite{BRComputing}, \cite{HibAlgebraic}, \cite[Chapter 12]{MSCombinatorial}). 

We first consider the permutation characters $\{ \chi_{mP} \}_{m \ge 0}$. For any positive integer $m$, let $\chi^*_{mP}$ denote the permutation character corresponding to the action of $G$ on the interior lattice points $\Int(mP) \cap M'$ in $mP$. The theorem below is due to Ehrhart in the case when $G = 1$ \cite{EhrLinearI}. 

%generalizes Ehrhart's original results in  \cite{EhrLinearI}, in the case when $G$ is trivial. 

\begin{ntheorem}[Theorem~\ref{t:combinatorics}]%\label{t:combinatorics}
The function $L(m) = \chi_{mP} \in R(G)$ is a  quasi-polynomial in $m$ of degree $d$ and period dividing the  exponent of $G$. Moreover, $L(0)$ is the trivial character, and $(-1)^{d} L(-m) = \chi^*_{mP} \cdot \det (\rho)$ for any positive integer $m$.  
\end{ntheorem}

As an application, for any positive integer $m$, let $f_{P/G}(m)$ (respectively $f_{P/G}^\circ(m)$) denote the number of $G$-orbits of $mP \cap M'$ (respectively $\Int(mP) \cap M'$). Similarly, let $\tilde{f}_{P/G}(m)$ (respectively $\tilde{f}_{P/G}^\circ(m)$) denote the number of $G$-orbits of $mP \cap M'$ (respectively $\Int(mP) \cap M'$) whose isotropy subgroup is contained in $\{ g \in G \mid \det(\rho(g)) = 1 \}$. 
By computing multiplicities of the trivial character and $\det(\rho)$ in $\chi_{mP}$, we deduce the following corollary. 

\begin{ncor}[Corollary~\ref{c:orbits}]
With the notation above, $f_{P/G}(m)$ and $\tilde{f}_{P/G}(m)$ are  quasi-polynomials in $m$ of degree $d$, with leading coefficient $\frac{\vol P}{|G|}$ and period dividing the  exponent of $G$. Moreover, $f_{P/G}(0) = \tilde{f}_{P/G}(0) = 1$, and 
$(-1)^{d} f_{P/G}(-m) =  \tilde{f}_{P/G}^\circ(m)$ and $(-1)^{d} \tilde{f}_{P/G}(-m) =  f^\circ_{P/G}(m)$ for any positive integer $m$. 
\end{ncor}

For example, if $G = \Sym_{n}$ acts on $\Z^n$ by permuting coordinates, and $P$ is the standard simplex with vertices  
$\{ e_1, \ldots, e_n \}$, then  $f_{P/G}(m)$ equals the number of partitions of $m$ with at most $n$ parts, and $\tilde{f}^\circ_{P/G}(m)$ equals the number of partitions of $m$ with $n$ distinct parts. In this case, 
the reciprocity result above is a classical result on partitions \cite[Theorem~4.5.7]{StaEnumerative}.

%follows from  in 

%add easy example with $P$ equal to $(d - 1)$-dimensional simplex in $\Z^d$ and $G = \Sym_d$. In this case, $f_{P/G}(m)$ equals the number of partitions of $m$ with at most $d$ parts, and $\tilde{f}^\circ_{P/G}(m)$ equals the number of partitions of $m$ with $d$ distinct parts. The reciprocity result follows from \cite[Theorem~4.5.7]{StaEnumerative}.

With the notation of the above theorem, we may write 
\[
L(m) = L_d(m) m^d + L_{d - 1}(m) m^{d - 1} + \cdots + L_0(m),
\]
where $L_i(m) \in R(G)$ is a periodic function in $m$. % with period dividing the exponent of $G$.
We prove that the leading coefficient equals $L_d = L_d(m) = \frac{\vol P}{|G|} \chi_{\st}$, where 
$\chi_{\st}$ is the character associated to the \emph{standard representation} of $G$ (Corollary~\ref{c:leading}). The latter fact may also be deduced from the work of Howe in \cite{HowAsymptotics} (Remark~\ref{r:howe}). 
We give a complete description of %the character of
 %leading term $L_d = L_d(m)$ (Corollary~\ref{c:leading}) and 
 the second leading coefficient $L_{d - 1}(m)$, and show that its period divides $2$ (Corollary~\ref{c:second}). %of this quasi-polynomial. 
 %In particular, we show that $L_{d - 1}(m)$ has period dividing $2$, and, 
 Moreover, if $\phi[t]$ is a polynomial, then we prove that $L_{d - 1} = L_{d - 1}(m)$ is independent of $m$ (Remark~\ref{r:second}).

We next consider the power series $\phi[t]$. 
 We show that $\phi[t](g)$ is a rational function in $t$ that is regular at $t = 1$ (Lemma~\ref{l:rational}), and give a complete description of the corresponding rational class function %$\tilde{\phi} = 
 $\phi[1]$ (Proposition~\ref{p:rational}). %We deduce the following corollary of the above theorem. 
The \define{degree} $s$ of $P$ is the degree of $h^*(t)$, and the \define{codegree} of $P$ equals 
$l = d + 1 - s$. 
 
 \begin{ncor}[Corollary~\ref{c:reciprocity}]
With the notation above, 
\begin{equation*}%\label{e:interior}
\sum_{m \ge 1} \chi^*_{mP} t^m = \frac{ t^{d + 1}\phi[t^{-1}]}
%{ (-1)^{d + 1}(t  - 1)(t^d - V \, t^{d - 1} + \bigwedge^2 V \,  t^{d - 2}  - \cdots + (-1)^d \bigwedge^d V)}
%\bigwedge^d V - \bigwedge^{d - 1} V t + \cdots + (-1)^{d - 1} V t^{d - 1}+ (-1)^d )}
{(1- t)\det[I - \rho t]}.
\end{equation*}
In particular, if $\phi[t]$ is a polynomial, then $\phi[t]$ has degree $s$ and $\phi_s = \chi^*_{lP}$. 
%the degree $s$ of $\phi[t]$ is equal to the degree of $P$, and $\phi_s = \chi^*_{lP}$. 
 %In particular, the multiplicity of the trivial representation in $\phi_s$ equals the number of $G$-orbits of 
%$\Int(lP) \cap M$.
% Moreover, the multiplicity of the trivial representation in $\phi_s$ is at least the number of $G$-orbits of 
%$\Int(lP) \cap M$.
% the 
 %trivial representation  occurs with non-zero multiplicity in 
\end{ncor}
We deduce that $\phi[t] = t^s\phi[t^{-1}]$ if and only if $lP$ is a translate of a reflexive polytope (Corollary~\ref{c:reflexive}), generalizing a result of Stanley in the case when $G = 1$ \cite[Theorem~4.4]{StaHilbert2}. 
 We also describe the behavior of $\{ \chi_{mP} \}_{m \ge 0}$ and $\phi[t]$ under the operations of \emph{direct product} and \emph{direct sum}, and prove an equivariant generalization of a theorem of 
 Braun \cite[Theorem~1]{BraEhrhart}. More specifically, 
we prove  that  if $P$ is a $G$-invariant reflexive polytope and $Q$ is an $H$-invariant lattice polytope containing the origin in its interior, then $\phi_{P \oplus Q}[t] = \phi_P[t] \cdot \phi_Q[t]$ (Proposition~\ref{p:Braun}).

 We next consider the delicate question of  when $\phi[t]$ is a polynomial. 
 In Section~\ref{s:effectiveness}, we provide distinct criterion that guarantee either that $\phi[t]$ is a polynomial (Lemma~\ref{l:polynomial}), or is not a polynomial (Lemma~\ref{l:bad}). We say that 
 $\phi[t]$ is \define{effective} if each virtual character $\phi_i$ is a character. Clearly, if $\phi[t]$ is effective, then $\phi[t]$ is a polynomial. We prove that $\phi_1$ is a character (Corollary~\ref{c:basic}), 
 and if  $P$ is a \define{simplex} (i.e. $P$ has $d + 1$ vertices),  then we show that the $\phi_i$ are explicit permutation representations (Proposition~\ref{p:simplex}).
 
 We offer the following conjecture. If $Y$ denotes the toric variety corresponding to $P$ with corresponding ample, torus-invariant line bundle 
$L$, then one may ask whether $(Y,L)$ admits a $G$-invariant hypersurface that is \define{non-degenerate} in the sense of Khovanski{\u\i} \cite{KhoNewton}. We refer the reader to Section~\ref{s:effectiveness} for details.

\begin{nconj}[Conjecture~\ref{c:big}]
With the notation above, the following conditions are equivalent
\begin{itemize}
\item $(Y, L)$ admits a $G$-invariant non-degenerate hypersurface, % with Newton polytope $P$
\item $\phi[t]$ is effective,
\item  $\phi[t]$ is a polynomial. 
\end{itemize}
\end{nconj}

The fact that the first condition implies the second condition is Theorem~\ref{t:nondeg}, and is proved in \cite{YoRepresentations} by realizing $\phi_{i + 1} \cdot \det(\rho)$ as the character associated to the action of $G$ on the $i^\mth$ graded piece of the
Hodge filtration on the primitive part of the middle cohomology (with compact support) of a $G$-invariant non-degenerate hypersurface.  In fact, this result provided the initial motivation for this project. 
The following corollary can be deduced from this result using Bertini's theorem 
\cite[Corollary~10.9]{HarAlgebraic}.
% (see Theorem~\ref{t:nondeg} and Corollary~\ref{c:fixed} respectively). 
 
 \begin{ncor}(Theorem~\ref{t:nondeg})
Let 
$\Gamma(Y, L)^G \subseteq \Gamma(Y, L)$ denote the sub-linear system of $G$-invariant global sections of $L$. If $\Gamma(Y, L)^G$ is base point free, then $\phi[t]$ is effective. 
\end{ncor}
 
%If $H$ is a finite group acting on a lattice $N$, and $Q$ is an $H$-invariant lattice polytope, then 
%the direct product $P \times Q$ and the direct sum $P \oplus Q$ are $(G \times H)$-invariant lattice polytopes. We show that $\chi_{m(P \times Q)} = \chi_{mP} \cdot \chi_{mQ}$ (Lemma~\ref{}), and that if $P$ is a reflexive polytope and $Q$ contains the origin in its interior, then $\phi[P \oplus Q; t] = \phi[P;t] \cdot \phi[Q; t]$ (Proposition~\ref{}),
%generalizing a result of Braun  \cite[Theorem~1]{BraEhrhart}. 

One easily deduces the following useful combinatorial criterion for effectiveness. 

\begin{ncor}[Corollary~\ref{c:fixed}]
If every  face $Q$ of $P$ with $\dim Q > 1$ contains a lattice point %in its relative interior  
that is $G_Q$-fixed, where $G_Q$ denotes the stabilizer of $Q$,  then 
$\phi[t]$ is effective. 
\end{ncor}

In particular, if $\dim P = 2$ and $P$ contains a $G$-fixed lattice point, then $\phi[t]$ is effective (Corollary~\ref{c:dim2}), and if the order of $G$ divides $m$, then $\phi_{mP}[t]$ is effective (Corollary~\ref{c:multiple}). 

Finally, we consider applications and examples of this theory. 
Firstly, we  prove that if $P$ contains the origin and the fan over its faces can be refined to a smooth, $G$-invariant fan $\triangle$ such that the  primitive integer vectors of the rays of $\triangle$ coincide with the non-zero vertices of $P$, then $\phi[t]$ coincides with the character of the representation of $G$ on the cohomology $H^*(X, \C)$ of the associated toric variety $X = X(\triangle)$  (Proposition~\ref{p:cohomology}). If $\triangle$ is the Coxeter fan associated to a root system, then we recover a formula of Procesi, Dolgachev and Lunts, and Stembridge \cite[Theorem~1.4]{SteSome}  for the character associated to the action of the Weyl group on  $H^*(X, \C)$ (Corollary~\ref{c:Weyl}). These characters have been studied by
Procesi~\cite{ProToric}, Stanley \cite[p. 529]{StaLog}, Dolgachev, Lunts~\cite{DLCharacter}, Stembridge~\cite{SteSome, SteEulerian} and Lehrer~\cite{LehRational}. In the type $A$ case, we show that we may also realize $\phi[t]$ from the action of $G = \Sym_d$ on the hypercube $P = [0,1]^d$ (Lemma~\ref{l:typeA}). 
In this case, we use results of Stembridge in \cite{SteEulerian} to give an explicit description of $\phi[t]$ in  terms of \define{marked tableaux} (Proposition~\ref{p:cube}). 

Observe that whenever $\phi[t]$ is effective, we obtain a refinement of the $h^*$-polynomial by considering dimensions of isotypic components of $\phi[t]$. In the case of the hypercube above, we recover Stembridge's refinement of the Eulerian numbers (Remark~\ref{r:refinement}). 

Secondly, observe that 
%These representations are interesting from a combinatorial perspective because they
the characters $\{ \chi_{mP} \}_{m \ge 0}$ 
 encode the Ehrhart theory of the rational polytopes $P_g = \{ u \in P \mid g \cdot u = u \}$ for all $g$ in $G$. 
More specifically, %the value of the character of $\chi_{mP}$ at $g \in G$ is given by 
$\chi_{mP}(g) = f_{P_g}(m) := \# (mP_g \cap M')$ (Lemma~\ref{l:character}), where $f_{P_g}(m)$ is a quasi-polynomial called the \define{Ehrhart quasi-polynomial} of $P_g$.  
When $P$ is a simplex, we deduce a formula for the generating series of  $f_{P_g}(m)$
(Proposition~\ref{p:simplex}). 
A \define{pseudo-integral polytope} is a rational polytope whose Ehrhart quasi-polynomial is a polynomial %, and these polytopes have recently been studied by 
(see the work of De Loera, McAllister and Woods \cite{DLMVertices, MWMinimum}), and we apply this formula to construct new pseudo-integral polytopes in all dimensions in Section~\ref{s:pip}. 

Lastly, it follows from Corollary~\ref{c:fixed} that  if $G = \Z/2\Z$ and $P$ is a \define{centrally symmetric polytope}, then $\phi[t]$ is effective. 
In Section~\ref{s:central}, we show that this fact is equivalent to the lower bounds $h_i^* \ge \binom{d}{i}$ on the coefficients of the $h^*$-polynomial of a centrally symmetric polytope,  that were proved 
by
 Bey, Henk and Wills in \cite[Remark 1.6]{BHWNotes}.
% We compute $\phi[t]$ when $G = \Sym_d \times \, \Z/2\Z$ acts on the 
 %lattice polytope with vertices 
  %$\{ \pm e_i,  \; \pm (e_1 + \cdots + e_{d}) \mid 1 \le i \le d \}$ (Corollary~\ref{}). The latter polytopes are known as  \define{Klyachko-Voskresenski{\u\i} polytopes}. 
  We give
  an explicit description of $\phi[t]$ for all non-singular, centrally symmetric, reflexive polytopes (Proposition~\ref{p:centrally}), using a classification result of Klyachko  and  Voskresenski{\u\i}
  \cite{KVToroidal}.

%%%%%%%%%%%%%%%%%%%%%%%%%%%%%%%%%%%%%%%%%%%%%%%

We end the introduction with a brief outline of the contents of the paper. In Section~\ref{s:basicEhrhart} and Section~\ref{s:representation} we recall some basic facts about Ehrhart theory and representation theory respectively. In Section~\ref{s2} we reduce to the case when $\dim P + 1 = \dim M'_\R$, and provide the setup for the rest of the paper. In Section~\ref{s:Ehrhart} and Section~\ref{s:hstar} we prove our results on the representations $\{ \chi_{mP} \mid m \ge 0 \}$ and $\phi[t]$ respectively. In Section~\ref{s:effectiveness} we give criterion to determine whether or not $\phi[t]$ is effective, and 
in Section~\ref{s:cohomology} we give a geometric interpretation of $\phi[t]$ for a special class of polytopes. In Section~\ref{s:hypercube} and Section~\ref{s:central} we present examples when $P$ is a hypercube and a centrally symmetric polytope respectively. 
In Section~\ref{s:pip} we demonstrate how our results can be applied to compute the Ehrhart quasi-polynomials of certain rational polytopes. Finally, in Section~\ref{s:open} we present some open questions and conjectures.

\medskip

\noindent
{\it Notation and conventions.}  
All representations will be defined over $\C$. We often identify a representation $\chi$ with its associated character and write $\chi(g)$ for the evaluation of the character of $\chi$ at $g \in G$. 
If $V$ is a $\Z$-module, then we write $V_\R = V \otimes_\Z \R$ and $V_\C = V \otimes_\Z \C$. 

\medskip
\noindent
{\it Acknowledgements.}  
The author would like to thank Dave Anderson, Kalle Karu, Gus Lehrer, Mircea Musta\c t\v a, 
Benjamin Nill,  Roberto Paoletti,
John Stembridge, Geordie Williamson and Josephine Yu for helpful comments and enlightening discussions. In particular, Remark~\ref{r:Dave} is due to Dave Anderson, and Remark~\ref{r:Karu} is due to Kalle Karu.

%%%%%%%%%%%%%%%%%%%%%%%%%%%%%%%%%%%%%%%%%%%%%%%
%%%%%%%%%%%%%%%%%%%%%%%%%%%%%%%%%%%%%%%%%%%%%%%
%%%%%%%%%%%%%%%%%%%%%%%%%%%%%%%%%%%%%%%%%%%%%%%
%%%%%%%%%%%%%%%%%%%%%%%%%%%%%%%%%%%%%%%%%%%%%%%
%%%%%%%%%%%%%%%%%%%%%%%%%%%%%%%%%%%%%%%%%%%%%%%
\section{Ehrhart theory}\label{s:basicEhrhart}
%%%%%%%%%%%%%%%%%%%%%%%%%%%%%%%%%%%%%%%%%%%%%%%
%%%%%%%%%%%%%%%%%%%%%%%%%%%%%%%%%%%%%%%%%%%%%%%
%%%%%%%%%%%%%%%%%%%%%%%%%%%%%%%%%%%%%%%%%%%%%%%
%%%%%%%%%%%%%%%%%%%%%%%%%%%%%%%%%%%%%%%%%%%%%%%

In this section, we recall some basic facts about Ehrhart theory, and refer the reader to \cite{BRComputing} and \cite{HibAlgebraic} for introductions to the subject. 

%proofs of the statements below. 

Let $M$ be a lattice and let $Q \subseteq M_\R$ be a rational $d$-dimensional polytope. 
The \define{denominator} of $Q$ is the smallest positive integer $m$ such that $mQ$ is a lattice polytope. 
If we let $f_Q(m) = \#(mQ \cap M)$ for any positive integer $m$, then a classical result of Ehrhart \cite{EhrLinearI} asserts that $f_Q(m)$ is a quasi-polynomial of degree $d$, called the \define{Ehrhart quasi-polynomial} of $Q$. That is, there exists a positive integer $l$ and polynomials $g_0(m), \ldots, g_{l - 1}(m)$ such that $f_Q(m) = g_i(m)$ whenever $m \equiv i$ mod $l$. The minimal choice of such $l$ is called the \define{period} of the quasi-polynomial $f_Q(m)$. 
Ehrhart proved that
$f_Q(0) = 1$, and that the period of $f_Q(m)$ divides the denominator of $Q$.
Moreover, if we set $f^\circ_Q(m) = \# (\Int(mQ) \cap M)$ for any positive integer $m$, where $\Int(mQ)$ denotes the interior of $mQ$, then 
\[
f_Q(- m) = (-1)^d f^\circ_Q(m). 
\]
The latter result is known as \define{Ehrhart reciprocity}. The \define{index} $\ind(Q)$ of $Q$ is the smallest positive integer $m$ such that the affine span of $m Q$ contains a lattice point. With the notation above, the polynomial $g_i(m)$ is a polynomial of degree $d$ with leading coefficient equal to the  volume $\vol(Q)$ of $Q$ if $\ind(Q)$ divides $i$, and is identically zero otherwise.
Here $\vol(Q)$ equals the   Euclidean volume of $Q \subseteq \aff(Q)$ with respect to the lattice $\aff(Q) \cap M$, where 
$\aff(Q)$ denotes the affine span of $Q$.  
Alternatively, if
we write 
\[
f_Q(m) = c_d(m)m^d + c_{d - 1}(m) m^{d - 1} +  \cdots + c_0(m), 
\]
where $c_i(m)$ is a periodic function in $m$, % with period dividing the period of $Q$. 
then  $c_0(0) = 1$, and
%If $p_d$ is the smallest positive integer such that the affine span of $p_d Q$ contains a lattice point, and 
%$\vol Q$ denotes the Euclidean volume of $Q$, 
%%called the \define{$d$-index} of $Q$,
 %then 
\begin{equation}\label{e:dindex}
c_d(m) =   \left\{\begin{array}{cl} 
\vol(Q) & \text{if } \ind(Q) | m %;
 \\ 0 & \text{otherwise}. \end{array}\right. 
\end{equation}
If $Q$ is a lattice polytope, then $c_i(m) = c_i$ is constant, and $c_{d - 1}$ may be interpreted as 
half the (normalized) surface area of $Q$. Here the \emph{normalized surface area} of a facet $F$ of $Q$ equals the Euclidean volume of $F \subseteq \aff(F)$ with respect to the lattice $\aff(F) \cap M$.
%where 
%$\aff(F)$ denotes the affine span of $F$.  

 Let $P \subseteq M_\R$ be a $d$-dimensional lattice polytope.
 After possibly replacing $M_\R$ with the affine span of $P$, we may and will assume that $M$ has rank $d$. 
  It follows from the above result that
%Observe that if $Q$ is a lattice polytope, then
 $f_P(m)$ is a polynomial of degree $d$, called the \define{Ehrhart polynomial} of $P$. 
By a routine argument, it follows that its generating series has the form
\[
\sum_{m \ge 0} f_P(m) t^m = \frac{h^*(t)}{(1 - t)^{d + 1}},
\]
where $h^*(t) = h^*_P(t)  =  \sum_{i = 0}^d h_i^* t^i$ is a polynomial of degree at most $d$ with integer coefficients, 
called the \define{$h^*$-polynomial} of $P$.  Alternative names in the literature include 
$\delta$-polynomial and Ehrhart $h$-polynomial. 
%The sum of the coefficients $h^*(1)$ of the $h^*$-polynomial equals %the \define{normalized volume} $\Vol P = 
%$d!\vol P$. % of $P$. 
Ehrhart reciprocity translates into the following equality
\begin{equation}\label{hreciprocity}
\sum_{m \ge 1}  f^\circ_P(m) t^m = \frac{t^{d + 1}h^*(t^{-1})}{(1 - t)^{d + 1}}. 
\end{equation}
Observe that $h^*_0 = 1$, $h^*_1 = \# (P \cap M) - d - 1$ and $h_d^* =  \# (\Int(P) \cap M)$. Since $P$ has at least $d + 1$ vertices, we conclude that 
\begin{equation}\label{e:easyinequality}
0 \le h_d^* \le h_1^*. 
\end{equation}
In fact, Stanley used the theory of Cohen Macauley rings to prove that the coefficients $h_i^*$ are non-negative integers \cite{StaDecompositions}. A combinatorial proof was later given by Betke and McMullen in \cite{BMLattice}. 
The \define{degree} $s$ of $P$ is defined to be the degree of  $h^*(t)$, and the \define{codegree} $l$ of $P$ is defined by $l = d + 1 - s$. 
Ehrhart reciprocity implies that the codegree can be interpreted as $l = \min\{ m \mid \Int(mP) \cap M \ne \emptyset \}$, 
and the leading coefficient of $h^*(t)$ is given by $h^*_s = \#( \Int(lP) \cap M )$. 

The polytope $P$ is \define{reflexive} if the origin is its unique interior lattice point, and every non-zero lattice point in $M$ lies in the boundary of $mP$  for some positive integer $m$. 
The following theorem of Stanley was proved using commutative algebra, while a combinatorial proof was recently given by the author in \cite[Corollary~2.18]{YoInequalities}. 

\begin{theorem}\cite[Theorem 4.4]{StaHilbert2}\label{t:basicreflexive}
If $P$ is a lattice polytope of degree $s$ and codegree $l$, then the following are equivalent
\begin{itemize}
\item  $f^\circ_P(m) = f_P(m - l)$ for $m \ge l$,
\item  $h^*_{P}(t) = t^{s} h^*_{P}(t^{-1})$,
\item  $lP$ is a translate of a reflexive polytope. 
\end{itemize}

\end{theorem}

\begin{remark}\label{r:smooth}
Let $\triangle$ be a smooth, $d$-dimensional fan in $M_\R$, and let $|\triangle|$ denotes the support of $\triangle$. Let $\psi: |\triangle| \rightarrow \R$ be the piecewise linear function with respect to $\triangle$ that has value $1$ at the primitive lattice points of the rays of $\triangle$. If $P = \{ v \in M_\R \mid \psi(v) \le 1 \}$ is convex, then the $h^*$-polynomial of $P$ is equal to the \define{$h$-polynomial} of $\triangle$. That is, 
\[
h^*_P(t) = h_{\triangle}(t) =  \sum_{i = 0}^{d} f_{i} t^i (1 - t)^{d - i},
\]
where $f_i$ equals the number of cones in $\triangle$ of dimension $i$. 
Note that if $|\triangle| = M_\R$, then $P$ is reflexive. Also, one can define the $h^*$-polynomial of $P$ even if $P$ is not convex, and the above equality holds. 
\end{remark}

\excise{
We also consider the following result of Stanley (Theorem 4.4 \cite{StaHilbert2}), that was proved using commutative algebra. 
\begin{ntheorem}
If $P$ is a lattice polytope of degree $s$ and codegree $l$, then $\delta_{P}(t) = t^{s} \delta_{P}(t^{-1})$ if and only if 
$lP$ is a translate of a reflexive polytope.  
\end{ntheorem}
}

%%%%%%%%%%%%%%%%%%%%%%%%%%%%%%%%%%%%%%%%%%%%%%%
%%%%%%%%%%%%%%%%%%%%%%%%%%%%%%%%%%%%%%%%%%%%%%%
%%%%%%%%%%%%%%%%%%%%%%%%%%%%%%%%%%%%%%%%%%%%%%%
%%%%%%%%%%%%%%%%%%%%%%%%%%%%%%%%%%%%%%%%%%%%%%%
%%%%%%%%%%%%%%%%%%%%%%%%%%%%%%%%%%%%%%%%%%%%%%%
\section{Representation theory of finite groups}\label{s:representation}
%%%%%%%%%%%%%%%%%%%%%%%%%%%%%%%%%%%%%%%%%%%%%%%
%%%%%%%%%%%%%%%%%%%%%%%%%%%%%%%%%%%%%%%%%%%%%%%
%%%%%%%%%%%%%%%%%%%%%%%%%%%%%%%%%%%%%%%%%%%%%%%
%%%%%%%%%%%%%%%%%%%%%%%%%%%%%%%%%%%%%%%%%%%%%%%

In this section, we recall some basic facts about the representation theory of finite groups over the complex numbers. We refer the reader to \cite{FHRepresentation} and \cite{IsaCharacter} for an introduction to the subject and proofs of the statements below. 

If $G$ is a finite group,  then a (complex) \emph{representation}  of $G$ is a finite-dimensional complex vector space $V$ with a linear action $\rho: G \rightarrow GL(V)$ of $G$. We say that $V$ is \emph{irreducible} if it contains no non-trivial $G$-invariant subspaces. 
 Every representation is isomorphic to a direct sum of irreducible representations.
  If $W$ is an irreducible representation and   $V \cong \oplus V_i$, where each $V_i$ is irreducible, then the \emph{multiplicity} of $W$ in $V$ is the number of irreducible representations $V_i$ isomorphic to $W$. 
%The number of isomorphism classes of irreducible representations of $G$ is equal to the number of conjugacy classes of $G$. 
The %complex
 \define{representation ring} $R(G)$ is defined to be the quotient of the %$\C$-vector space
 free abelian group generated by isomorphism classes of $G$-representations
  %with basis indexed by the (isomorphism classes of) $G$-representations
   by the $\Z$-submodule generated by relations of the form $V \oplus W - V - W$. 
Addition (respectively multiplication) of classes of representations is given by taking direct sums (respectively tensor products) of representations. 
%Addition is given by direct sum, and multiplication is defined by $V \cdot W = V \otimes W$. 
%by tensor product of representations. 
%Clearly, $R(G)$ has a $\C$-basis indexed by the irreducible $G$-representations, and the trivial representation may be identified with $1 \in R(G)$.  
Elements of $R(G)$ are called \define{virtual representations}, and we let $1 \in R(G)$ denote the class of the trivial representation of $G$. An element of $R(G)$ is an \define{effective representation} if is equal to the class of a representation $V$ of $G$. 
%corresponds to a representation of $G$. 

The \emph{character} $\chi: G \rightarrow \C$ associated to a representation $V$ is the function $\chi(g) = \tr(\rho(g))$, where $\tr$ denotes the trace function. A character of a representation is a \emph{class function} i.e. a function from $G$ to $\C$ that is constant on conjugacy classes.  Addition and multiplication in $\C$ gives the set of all class functions $\C_{\class}(G)$ the structure of a $\C$-algebra. 
The $\C$-algebra homomorphism from $R(G) \otimes_\Z \C$ to $\C_{\class}(G)$, taking a representation to its character, is an isomorphism. 
The vector space  $\C_{\class}(G)$ admits a Hermitian inner product
\[
\left\langle \alpha, \beta \right\rangle = \frac{1}{|G|} \sum_{g \in G} \overline{\alpha(g)} \beta(g),
\]
where $|G|$ denotes the order of $G$, and $\overline{a}$ denotes the complex conjugate of $a \in \C$.  The characters of the irreducible representations of $G$ form an orthonormal basis of $\C_{\class}(G)$. In the remainder of the paper, we will often identify a representation with its character. 

 If $\bigwedge^m V$ and $\Sym^m V$ denote the exterior and symmetric powers of $V$ respectively, then we have the following (well-known) equality in $R(G)[[t]]$. 

\begin{lemma}\label{l:exterior}
Let $G$ be a finite group and let $V$ be an $r$-dimensional representation. Then 
\begin{equation*}%\label{e:exterior}
\sum_{m \ge 0} \Sym^m V t^m =  \frac{ 1}{1 - Vt + \wedge^2 V t^2 - \cdots + (-1)^r \wedge^r V t^r}.
\end{equation*}
Moreover, if an element $g \in G$ acts on $V$ via a matrix $A$, and if $I$ denotes the identity $r \times r$ matrix, then both sides equal $\frac{1}{\det(I - tA)}$ %$1/\det(I - tA)$ 
when the associated characters are evaluated at $g$. 
% (i.e. under the ring homomorphism
%%under the isomorphism 
%$\tr(g; \; ) : R(G)[[t]] \cong \C_{\class}(G)[[t]] \rightarrow \C[[t]]$). 
\end{lemma}
\begin{proof}
The following simple proof was related to me by John Stembridge. If an element $g \in G$ acts of $V$, then, since $g$ has finite order, we may assume, after a change of basis, that $g$ acts via a diagonal matrix
 $(\lambda_1,\ldots, \lambda_r)$. Then both sides of  the equation equal $\frac{1}{(1 - \lambda_1 t)\cdots(1 - \lambda_rt)}$ when evaluated at $g$. 
 % If we view \eqref{e:exterior} as an equality in the ring $\C_{\class}(G)[[t]]$, then both sides of \eqref{e:exterior} equal $\frac{1}{(1 - \lambda_1 t)\cdots(1 - \lambda_rt)}$ when evaluated at $g$. 
  %The result follows since the
 %character of both sides of the equation (i.e. the corresponding elements in $\C_{\class}(G)[[t]]$ evaluated at $g$ is 
 %$\frac{1}{(1 - \lambda_1 t)\cdots(1 - \lambda_rt)}$. 
\end{proof}

If $H$ is a subgroup of $G$, with group algebra $\C[H]$, and $W$ is an $H$-representation, then the \emph{induced representation} $\Ind_H^G W$ is the $G$-representation $\C[G] \otimes_{\C[H]} W$. 
If $W'$ is a representation of $G$, then we let $\Res_H^G W'$ denote 
the restriction of 
%the action of $G$ to $H$. 
$W'$ to an $H$-representation. 
Frobenius reciprocity states that for a $G$-character $\chi$ and an $H$-character $\phi$, 
\[
\left\langle  \Ind_H^G \chi, \; \phi  \right\rangle = \left\langle  \chi, \; \Res_H^G \phi  \right\rangle.  
\]
If $G$ acts transitively on a set $S$, then the associated \emph{isotropy group} $H$ is 
the subgroup of $G$ that fixes a given $s$ in $S$, and is well-defined up to conjugation. The corresponding permutation representation is isomorphic to the induced representation $\Ind_H^G 1$ of the trivial representation of $H$. We immediately deduce the following lemma.  

\begin{lemma}\label{l:permutation}
Suppose $G$ acts on a set $S$, and let $\chi$ denote the corresponding permutation character. Then 
$\chi(g)$ equals the number of elements of $S$ fixed by $g$ in $G$, and if $\lambda: G \rightarrow \C$ is a $1$-dimensional representation, then  
the multiplicity of $\lambda$ in $\chi$ is equal to the number of $G$-orbits of $S$
whose isotropy subgroup is contained in the subgroup $\lambda^{-1}(1)$ of $G$. 
%the associated permutation representation %$V$
 %is equal to the number of $G$-orbits of $S$. More generally, if $\lambda: G \rightarrow 
\end{lemma}

\excise{
We will also need the following lemma. 
\begin{lemma}\label{l:dual}
Suppose $G$ acts linearly on a lattice $N$ of rank $r$. Then we have isomorphisms of 
$G$-representations $\bigwedge^i N_\C \otimes \bigwedge^r N_\C \cong \bigwedge^{r - i} N_\C$.   
\end{lemma}
\begin{proof}
As in the previous proof, we may assume that $g \in G$ acts on $N_\C$ via a diagonal matrix
 $(\lambda_1,\ldots, \lambda_r)$, for some roots of unity $\lambda_i$. Note that $\lambda_i^{-1} = \overline{\lambda}_i$, and  
%Note that  %every element 
$g$ acts on $\bigwedge^r N_\C$ via multiplication by $\pm 1$, and hence 
$(\lambda_1\cdots \lambda_r)^2 = 1$. 
%In particular, $(\bigwedge^r N_\C)^2$ is the trivial representation. 
%Note that $(\lambda_1\cdots \lambda_r)^2 = 1$  Then 
We conclude that the left hand side evaluated at $g$ is equal to 
 \[
 \lambda_1\cdots \lambda_r \sum_{k_1 < \cdots < k_i} \lambda_{k_1}\cdots \lambda_{k_i} = 
 %\sum_{k'_1 < \cdots < k'_{ r- i}} \lambda_{k'_1}^{-1} \cdots \lambda_{k'_{r - i}}^{-1} = 
  \sum_{k'_1 < \cdots < k'_{ r- i}} \overline{\lambda}_{k'_1} \cdots \overline{\lambda}_{k'_{r - i}}
  =   \sum_{k'_1 < \cdots < k'_{ r- i}} \lambda_{k'_1} \cdots \lambda_{k'_{r - i}}. 
   \]
\end{proof}
}

\begin{example}[The symmetric group]\label{e:symmetric}
If $G = \Sym_d$ denotes the symmetric group on $d$ letters, then the irreducible representations $\chi^{\lambda}$ of $G$ 
%and their corresponding characters $\chi_\lambda$ 
are indexed by partitions $\lambda$ of $d$. For example, $\chi^{(d)}$ is the trivial  representation, $\chi^{(1^d)}$ is the \emph{sign} representation, and 
 $\chi^{(d -  1, 1)}$ is the \emph{reflection representation} corresponding to the standard action of $G$ on    $\C^d/\C(1,\ldots, 1)$. More generally, the hook partitions $\chi^{(d - r,1^{r})} =   \bigwedge^{r} \chi^{(d -  1, 1)}$ correspond to exterior powers of the reflection representation. 

\excise{
Let  $G = \Sym_n$ be the symmetric group on $n$ letters, and let $V_{\st} \cong \C^n$ denote the standard representation. 
The irreducible representations $\chi^{\lambda}$ of $G$ 
%and their corresponding characters $\chi_\lambda$ 
are indexed by partitions $\lambda$ of $n$. 
 For example, $\chi^{(n)}$ is the trivial representation, $\chi^{(1^n)}$ is the sign representation, and 
 $\chi^{(n -  1, 1)}$ is the \emph{reflection representation} corresponding to the vector space  
 $V_{\st}/(1

 the partition $(n)$ corresponds to the trivial representation, the partition $(1,1,\ldots, 1)$ corresponds to the sign representation, and the partition $(n - 1, 1)$ corresponds the quotient $V_{\st}/\C(1,\ldots, 1)$.  More generally, %the partition 
 $V_{(r,1 ,\ldots, 1)} \cong \bigwedge^{n - r} V_{n - 1, 1}$ for $1 \le r \le n$. 
 }
\end{example}

%%%%%%%%%%%%%%%%%%%%%%%%%%%%%%%%%%%%%%%%%%%%%%%
%%%%%%%%%%%%%%%%%%%%%%%%%%%%%%%%%%%%%%%%%%%%%%%
%%%%%%%%%%%%%%%%%%%%%%%%%%%%%%%%%%%%%%%%%%%%%%%
%%%%%%%%%%%%%%%%%%%%%%%%%%%%%%%%%%%%%%%%%%%%%%%
\section{The setup}\label{s2} 
%%%%%%%%%%%%%%%%%%%%%%%%%%%%%%%%%%%%%%%%%%%%%%%
%%%%%%%%%%%%%%%%%%%%%%%%%%%%%%%%%%%%%%%%%%%%%%%
%%%%%%%%%%%%%%%%%%%%%%%%%%%%%%%%%%%%%%%%%%%%%%%
%%%%%%%%%%%%%%%%%%%%%%%%%%%%%%%%%%%%%%%%%%%%%%%

Recall from the introduction that $G$ is  a finite group acting linearly on a lattice $M' \cong \Z^n$, and $P$ is a $d$-dimensional $G$-invariant lattice polytope. In this section, we explain how one can always reduce to the case when %$n = d + 1$, 
$M' = M \oplus \Z$ for some lattice $M$ of rank $d$ and $P \subseteq M \times 1$. 
We also show that one may equivalently consider $d$-dimensional lattice polytopes in lattices of rank $d$, that are $G$-invariant `up to translation'. The setup deduced at the end of the section will be used throughout the paper. 
%Alternatively, we show that we may equivalently consider the case when $n = d$ if we allow $P$ to be $G$-invariant `up to translation'.  

%\begin{remark}\label{r:affine}
Observe that the affine span $W$ of $P$ in $M'_\R$ is $G$-invariant. If we fix a lattice point $\bar{u} \in W \cap M'$, then $M := W \cap M' - \bar{u}$ has the structure of a lattice of rank $d$ and $G$ acts linearly on $M$ via
\[
g \cdot (u - \bar{u})  = gu - g\bar{u} = (gu  - g\bar{u} + \bar{u}) - \bar{u},
\] 
for all $g \in G$ and $u \in W \cap M'$. Regarding $P$ as a lattice polytope in $M$, we see that $P$ is 
invariant under $G$ `up to translation'. That is, if we set consider the function $w: G \rightarrow M$ defined by $w(g) = g\bar{u} - \bar{u}$, then $w(1) = 0$,  $w(gh) = w(g) + g \cdot w(h)$, and if we identify $P$ with the lattice polytope $P - \bar{u}$ in $M$, then  $g \cdot P = P - w(g)$ in $M$ for all $g \in G$. 

%\end{remark}

%\begin{remark}\label{r:reduction}
Conversely, assume that $G$ acts linearly on a $d$-dimensional lattice $M$, and $P$ is a $d$-dimensional lattice polytope that is invariant under $G$ `up to translation'.  That is, assume there exists a function $w : G \rightarrow M$ satisfying $w(1) = 0$ and $w(gh) = w(g) + g \cdot w(h)$, and such that $g \cdot P = P - w(g)$ for all $g \in G$. Then  $G$ acts linearly on the lattice $M' = M \oplus \Z$ as follows:  $g \cdot (u, \lambda) = (g \cdot u - \lambda w(g), \lambda)$ for any $g \in G$ and $(u,\lambda) \in M'$. If we identify $P$ with the lattice polytope $P \times  1$ in $M'$, then $P$ is invariant under the action of $G$.  
Note that we recover the original linear action of $G$ on $M$ and the induced action on $P$ `up to translation' via the action of $G$ on $M \times  0  \subseteq M'$ and $P \times  0$ respectively. Moreover, the complex $G$-representation $(M')_\C$ %= M' \otimes_\Z \C$ 
is isomorphic to $M_\C \oplus \C$, where $\C$ denotes the trivial representation. 
%\end{remark}

The preceding discussion motivates the following \define{setup}:

%\emph{Let $G$ be a finite group acting linearly on a lattice $M$ of rank $d$ via $\rho: G \rightarrow GL(M)$, and consider the action of $G$ on $M' = M \oplus \Z$ via the trivial action on the second coordinate

%\begin{center}
\emph{Let $G$ be a finite group acting linearly on a lattice $M' = M \oplus \Z$ of rank $d + 1$ such that the projection $M' \rightarrow \Z$ is equivariant with respect to the trivial action of $G$ on $\Z$. Let $P \subseteq M_\R \times  1$ be a $G$-invariant, $d$-dimensional lattice polytope.  
}

\emph{
By identifying $M$ with $M \times 0$, we regard $M$ as a lattice with a linear $G$-action
$\rho: G \rightarrow GL(M)$, and consider the corresponding complex $G$-representation  $M_\C$. 
We often identify $P$ with the lattice polytope $\{ u \in M_\R \mid u \times 1 \in P \}$ in $M_\R$, that is $G$-invariant `up to translation'. } %= M \otimes_\Z \C$.    }

%%%%%%%%%%%%%%%%%%%%%%%%%%%%%%%%%%%%%%%%%%%%%%%
%%%%%%%%%%%%%%%%%%%%%%%%%%%%%%%%%%%%%%%%%%%%%%%
%%%%%%%%%%%%%%%%%%%%%%%%%%%%%%%%%%%%%%%%%%%%%%%
%%%%%%%%%%%%%%%%%%%%%%%%%%%%%%%%%%%%%%%%%%%%%%%
\section{Equivariant Ehrhart theory}\label{s:Ehrhart}
%%%%%%%%%%%%%%%%%%%%%%%%%%%%%%%%%%%%%%%%%%%%%%%
%%%%%%%%%%%%%%%%%%%%%%%%%%%%%%%%%%%%%%%%%%%%%%%
%%%%%%%%%%%%%%%%%%%%%%%%%%%%%%%%%%%%%%%%%%%%%%%
%%%%%%%%%%%%%%%%%%%%%%%%%%%%%%%%%%%%%%%%%%%%%%%

The goal of this section is to study the permutation characters $\{ \chi_{mP} \}_{m \ge 0}$, and establish equivariant analogues of  
Ehrhart's original results (see Section~\ref{s:basicEhrhart}). Throughout the paper, we often 
identify representations with their characters. 
%We prove the main results of this paper in this section, and establish 

We will continue with the setup of Section~\ref{s2} above. %and assume that $G$ is finite group 
%That is, let $G$ be a finite group acting linearly on a lattice $M' = M \oplus \Z$ of rank $d + 1$ such that the projection $M' \rightarrow \Z$ is equivariant with respect to the trivial action of $G$ on $\Z$, and let $P \subseteq M_\R \times 1$ be a $G$-invariant, $d$-dimensional lattice polytope. 
We often abuse notation, and consider $P$ as a polytope in $M_\R$. %\comment{Is this needed?}
Recall that for any positive integer $m$, 
$\chi_{mP}$ (respectively $\chi^*_{mP}$)
denotes 
%the permutation representation of $G$ with $\C$-basis $\{ x^u \mid u \in mP \cap M'\}$ and 
%$g \cdot x^u = x^{gu}$ for all $g \in G$ and $u \in mP \cap M'$, and
%denote 
the complex permutation representation induced by the action of $G$ on the lattice points $mP \cap M$ (respectively $\Int(mP) \cap M$), and
 $\chi_{mP}$ denotes the trivial representation when $m = 0$.  

%Our first goal is to prove an equivariant version of Ehrhart's original results (see Section~\ref{s:basicEhrhart}). 
%Our first goal is to study the following rational polytopes. 
Consider the following rational polytopes. 

\begin{definition}
For any $g \in G$, let $P_g = \{ u \in P \mid g \cdot u = u \} \subseteq M_\R$. 
%if $(M'_\R)_g$ denotes the linear subspace of $M'_\R$ fixed by $g$, then let $P_g$ denote the rational polytope $P \cap (M'_\R)_g$. 
\end{definition}

The following simple lemma provides the motivation to consider these polytopes. 
Recall from Section~\ref{s:basicEhrhart} that $f_{P_{g}}(m) = \#(mP_g \cap M)$ is the Ehrhart quasi-polynomial of $P_g$, and $f^\circ_{P_{g}}(m) = \#(\Int (mP_g) \cap M)$ for $m \ge 1$. 

\begin{lemma}\label{l:character}
For any positive integer $m$, $\chi_{mP}(g) = f_{P_g}(m)$ and $\chi_{mP}^*(g) = f_{P_g}^\circ(m)$. 
Also, $\chi_{mP}(g) = f_{P_g}(m) = 1$ when $m = 0$. 
\end{lemma} 
\begin{proof}
Since $\chi_{mP}$ is a permutation representation, $\chi_{mP}(g)$ is equal to the number of lattice points in $mP$ fixed by $g$ (Lemma~\ref{l:permutation}). The latter is equal to $f_{P_g}(m)$. The rest of the lemma follows similarly. 
\end{proof}

Let $M^g$ denote the subspace of $M_\R$ fixed by $g$. 
Note that, if we fix an isomorphism $M \cong \Z^d$, then $g$ acts on $M$ via an integer-valued matrix $A$, and $\dim M^g$ equals the number of times $1$ occurs (with multiplicity) as an eigenvalue of $A$. 
%Fix an element $g \in G$,  isomorphism $M \cong \Z^d$ and consider 

\begin{lemma}\label{l:dimension}
With the notation above, $\dim P_g = \dim M^g$.  
%equals the number of times $1$ occurs (with multiplicity) as an eigenvalue of $A$. 
\end{lemma}
\begin{proof}
%Let $\lambda$ denote the number of times $1$ occurs (with multiplicity) as an eigenvalue of $A$. 
If $(M')^g$ denotes the linear subspace of $M'_\R$ fixed by $g$, then, by definition, $P_g$ is the rational polytope $P \cap (M')^g$. Note that $(M')^g$ intersects $M_\R \times 1$ in an affine subspace of dimension $\dim M^g$, and hence we only need to show that $(M')^g \cap \Int(P) \ne \emptyset$. On the other hand,  if $\{ v_i \mid i \in I \}$ denotes the vertices of $P$, then
$\frac{1}{|I|} \sum_{i \in I} v_i$ is a $G$-invariant point in the interior of $P$.  
\end{proof}

If we fix an element $g \in G$, then $g$ permutes the 
 set of vertices $\{ v_i \mid i \in I \}$ of $P$.  Let $I / g$ denote the set of orbits of $I$ under the action of $g$, and, for each orbit $\iota \in I / g$, let $v_\iota = \frac{1}{|\iota|} \sum_{i \in \iota} v_i$ be the corresponding rational point in $P_g$. 
 %The \define{order} $\ord(g)$ of $g$ is the smallest positive integer $r$ such that $g^r = 1$.
 Recall that $P$ is a \emph{simplex} if it has precisely $d + 1$ vertices.

 \begin{lemma}\label{l:vertices}
 With the notation above, %the vertices of $P_g$ are contained in the set $\{ v_\iota  \mid \iota \in I / g \}$. 
 $P_g$  is the convex hull of $\{ v_\iota  \mid \iota \in I / g \}$.
 In particular, if $g^r$ fixes the vertices of $P$, then $rP_g$ is a lattice polytope.  
 %the denominator of $P_g$ divides the order of $g$. 
 Moreover, if $P$ is a simplex, then $P_g$ is a simplex with (distinct) vertices
 $\{ v_\iota  \mid \iota \in I / g \}$. % indexed by the $g$-orbits of vertices of $P$.
 %indexed by the orbits $I/g$.
% and $\dim P_g + 1$ equal to the number of $g$-orbits of vertices of $P$. 
 \end{lemma}
\begin{proof}
Any element $w \in P$ can be written in the form $w = \sum_{i \in I} \lambda_i v_i$ for some $\lambda_i \ge 0$ satisfying $\sum_{i \in I} \lambda_i = 1$. If $w \in P_g$ and 
$g^r$ fixes the vertices of $P$, 
%$r = \ord(g)$, %and $r$ denotes the order of $g$,
then 
\[
w = \frac{1}{r} \sum_{i = 0}^{r - 1} g^i w = \sum_{\iota \in I / g} \frac{1}{|\iota|} \sum_{i \in \iota} \lambda_i \sum_{i \in \iota} v_i =  \sum_{\iota \in I / g} (\sum_{i \in \iota} \lambda_i) v_\iota. 
\]
Since $\sum_{\iota \in I / g} (\sum_{i \in \iota} \lambda_i) = 1$, this 
%$\lambda_i = \lambda_j$ if $i$ and $j$ lie in the same orbit of $g$, and hence $w = \sum_{\iota \in I / g} |%\iota| \lambda_i v_\iota$, where 
%$i$ is any representative of $\iota$. Since $\sum_{\iota \in I / g} |\iota| \lambda_i  = 1$, this 
proves the first statement. Observing that $|\iota|$ divides $r$, %the order of $g$, 
we obtain the second statement. 
Finally, if $P$ is a simplex, then the vertices $\{ v_i \mid i \in I \}$ of $P$ 
%assume that $P$ is a simplex with vertices $\{ v_0, \ldots, v_d \}$. 
%Since the vertices $\{ v_0, \ldots, v_d \}$ 
form a basis of $M'_\C$. Hence we may identify  the representation $M'_\C$ with the permutation representation induced by the action of $G$ on 
 $\{ v_i \mid i \in I \}$. In particular, the dimension of the subspace $(M')^g_\C$ of $M'_\C$ fixed by $g$ is precisely the number of orbits $\iota \in I / g$. By Lemma~\ref{l:dimension}, we conclude that $\dim P_g + 1$ equals the number of $g$-orbits of vertices of $P$,
 % = 
%\# \{ v_\iota  \mid \iota \in I / g \}$, 
and the result follows.  
\end{proof}

Recall that $G$ acts on $M$ via $\rho: G \rightarrow GL(M)$. 
%Recall that $\det(\rho)$ denotes the $1$-dimensional representation $\bigwedge^d M_\C$.

\begin{lemma}\label{l:determinant}
With the notation above, the function %$G \rightarrow \C$, 
$g \mapsto (-1)^{d - \dim P_g}$ equals the representation $\det(\rho)$. 
% character of $\bigwedge^d M_\C$. 
\end{lemma}
\begin{proof}
If $g$ acts on $M$ via an integer-valued matrix $A$, then 
the eigenvalues of $A$ are roots of unity, and hence %the characteristic polynomial of $A$ has the form 
\[
\det(tI - A) = (t - 1)^a (t + 1)^b \prod_\zeta (t - \zeta)(t - \bar{\zeta}), 
\]
for some complex roots of unity $\zeta$. Comparing constant terms on both sides yields 
$(-1)^d \det(\rho(g)) = (-1)^a$. The result now follows from Lemma~\ref{l:dimension}. 

\excise{
%The characteristic polynomial 
$\det(tI - A) \in \Z[t]$ of $A$ has constant term $(-1)^d \det(\rho(g))$. Moreover, since $g$ has finite order, each irreducible factor of $\det(tI - A)$ divides the polynomial $t^N - 1$ for some positive integer $N$. The polynomial $t^N - 1$ has a unique %prime 
decomposition in $\Z[t]$
\begin{equation}\label{e:cyclotomic}
t^N - 1 = \coprod_{l | N} \Phi_l(t),
\end{equation}
where $\Phi_l(t)$ is the monic polynomial in $\Z[t]$ with complex roots (each with multiplicity $1$) given by the primitive $l^{\textrm{th}}$ roots of unity.  The polynomial $\Phi_l(t)$ is called the  $l^{\textrm{th}}$ \emph{cyclotomic polynomial} and it is irreducible in $\Z[t]$ \cite[Section~53]{VanModern}. We deduce that 
every irreducible factor of $\det(tI - A)$ in $\Z[t]$ is a cyclotomic polynomial. Moreover, it follows by induction on $l$ and \eqref{e:cyclotomic}, that  $\Phi_l(0) = -1$ if $l = 1$ and $\Phi_l(0) = 1$ if $l > 1$. Hence the constant term in $\det(tI - A)$ is equal to $(-1)^a$, where $a$ is the multiplicity of 
$(t - 1)$ in $\det(tI - A)$. Hence $a = (-1)^d \det(\rho(g))$ and the result
%The result now 
follows from Lemma~\ref{l:dimension}. 
}
\end{proof}

\begin{remark}
In fact, the above proof holds provided that $\rho: G \rightarrow GL(M_\R)$ is a real representation. 
\end{remark}

We are now ready to prove an equivariant analogue of Ehrhart's results in \cite{EhrLinearI}. 
Recall that the \emph{exponent} of $G$ is the smallest positive integer $N$ such that $g^N = 1$ for all $g \in G$,   %. Recall
and  that  the \emph{denominator} of a rational polytope $Q$ is the smallest positive integer $m$ such that $mQ$ is a lattice polytope.

\begin{theorem}\label{t:combinatorics}
Consider the function $L(m) = \chi_{mP} \in R(G)$ for any non-negative integer $m$. 
%If $L(m) = \chi_{mP}$ for any non-negative integer $m$, 
Then $L(m)$ is a  quasi-polynomial in $m$ of degree $d$ and period dividing the  exponent of $G$, 
and,
%Moreover, 
for any positive integer $m$,
$(-1)^{d} L(-m) = \chi^*_{mP} \cdot \det(\rho)$.  
\end{theorem}
\begin{proof}
%It will be enough to prove the statement at the level of characters. 
By Lemma~\ref{l:character}, $\chi_{mP}(g) = f_{P_g}(m)$ for any non-negative integer $m$, and, by Lemma~\ref{l:vertices} and Ehrhart's results (Section~\ref{s:basicEhrhart}), $f_{P_g}(m)$ is a quasi-polynomial of degree $\dim P_g$ with period dividing the exponent of $G$. Hence, by Ehrhart reciprocity and Lemma~\ref{l:character}, the character of 
$(-1)^{d} L(-m)$ evaluated at $g$ equals 
\[ (-1)^{d - \dim P_g} (-1)^{\dim P_g} f_{P_g}(-m) = (-1)^{d - \dim P_g} f_{P_g}^\circ(m) = (-1)^{d - \dim P_g} \chi_{mP}^*(g).
\]
 The result now follows from
 Lemma~\ref{l:determinant}.
\end{proof}

For any positive integer $m$, Lemma~\ref{l:permutation} implies that  $f_{P/G}(m) =  \left\langle \chi_{mP}, 1 \right\rangle$ (respectively $f_{P/G}^\circ(m)  = \left\langle \chi^*_{mP}, 1 \right\rangle$) equals the number of $G$-orbits of $mP \cap M$ (respectively $\Int(mP) \cap M$). Similarly, $\tilde{f}_{P/G}(m) = \left\langle \chi_{mP}, \det(\rho) \right\rangle$ (respectively $\tilde{f}_{P/G}^\circ(m)  = \left\langle \chi^*_{mP}, \det(\rho) \right\rangle$) equals the number of $G$-orbits of $mP \cap M$ (respectively $\Int(mP) \cap M$)  whose isotropy subgroup is contained in $\{ g \in G \mid \det(\rho(g)) = 1 \}$.

\begin{corollary}\label{c:orbits}
With the notation above, $f_{P/G}(m)$ and $\tilde{f}_{P/G}(m)$ are  quasi-polynomials in $m$ of degree $d$, with leading coefficient $\frac{\vol P}{|G|}$ and period dividing the  exponent of $G$. Moreover, $f_{P/G}(0) = \tilde{f}_{P/G}(0) = 1$, and 
$(-1)^{d} f_{P/G}(-m) =  \tilde{f}_{P/G}^\circ(m)$ and $(-1)^{d} \tilde{f}_{P/G}(-m) =  f^\circ_{P/G}(m)$ for any positive integer $m$. 
\end{corollary}
\begin{proof}
We apply Theorem~\ref{t:combinatorics} to the inner products $ \left\langle \chi_{mP}, 1 \right\rangle$ and  $\left\langle \chi_{mP}, \det(\rho) \right\rangle$, using the fact that $\chi_{mP}$ is the trivial representation when $m = 0$, and the fact that 
$\left\langle \chi^*_{mP} \cdot \det(\rho), 1 \right\rangle = \left\langle \chi^*_{mP} , \det(\rho) \right\rangle$. 
The statement about the leading coefficients of $f_{P/G}(m)$ and $\tilde{f}_{P/G}(m)$ follows from Corollary~\ref{c:leading} below. 
\end{proof}

With the notation of Theorem~\ref{t:combinatorics}, we may write 
\[
L(m) = L_d(m) m^d + L_{d - 1}(m) m^{d - 1} + \cdots + L_0(m),
\]
where $L_i(m) \in R(G)$ is a periodic function in $m$ with period dividing the exponent of $G$. Observe that $L_0(0)$ is the trivial representation. 
Below we give an explicit description of the two leading terms of this quasi-polynomial. Recall that the \emph{standard representation} $\chi_{\st}$ of $G$ is the permutation representation induced by the action of $G$ on itself by left multiplication. Since $g \ne 1$ has no fixed points with respect to this action, the character of $\chi_{\st}$ is given by 
%  representation associated to the action of $G$ on the group algebra $\C[G]$
\begin{equation}\label{e:standard}
\chi_{\st}(g) =  \left\{\begin{array}{cl} 
|G| & \text{if } g = 1 %;
 \\ 0 & \text{otherwise}. \end{array}\right. 
\end{equation}
It is a standard fact that if $\{ \chi_1, \ldots, \chi_r \}$ denote the irreducible representations of $G$, then 
$\chi_{\st} = \sum_{i = 1}^r \chi_i(1) \chi_i$. 

\begin{corollary}\label{c:leading}
With the notation above, the leading coefficient $L_d(m) = L_d$ is independent of $m$ and given by
$L_d = \frac{\vol P}{|G|} \chi_{\st}$. In particular, the multiplicity of a fixed irreducible representation $\chi$ in $\chi_{mP}$ is a quasi-polynomial in $m$ of degree $d$ with leading coefficient $\frac{\chi(1)\vol P}{|G|}$.  
\end{corollary}
\begin{proof}
By Lemma~\ref{l:character}, $\chi_{mP}(g) = f_{P_g}(m)$ is a quasi-polynomial of degree strictly less than $d$ unless $g = 1$, in which case, the leading term is $\vol(P)m^d$ (see Section~\ref{s:basicEhrhart}). The first statement now follows from \eqref{e:standard}. The second statement is immediate from the fact that the multiplicity $\left\langle \chi_{\st}, \chi \right\rangle$ of an irreducible representation $\chi$ in $\chi_{\st}$ is equal to its dimension $\chi(1)$.   
\end{proof}

\begin{remark}\label{r:howe}
If $C$ denotes the cone over $P$, then $\chi_{mP}$ may be viewed as the representation of $G$ on the $m^\mth$ graded piece of the semi-group algebra $R = \C[ C \cap M']$. 
From this perspective, the above corollary may be viewed as a special case of the results of Howe in \cite{HowAsymptotics}. We refer the reader to the work of Paoletti \cite[Theorem~1]{PaoAsymptotic} for a similar result on equivariant volumes of big line bundles on smooth, projective complex varieties. 
%For a more general result on equivariant volumes of big line bundles, see Paoletti has since proved a more general result in \cite[Theorem~1]{PaoAsymptotic}. 
\end{remark}

For any $g \in G$, recall that the \emph{index} $\ind(P_g)$ of $P_g$ is the smallest positive integer $m$ such that the affine span of $m P_g$ contains a lattice point. By Lemma~\ref{l:vertices}, the index of $P_g$ divides the order of $g$. 
If $g$ acts on $M$ via the integer-valued matrix $A$, then $g$ is a \define{reflection} if all but one of the eigenvalues of $A$ is equal to $1$ (the other eigenvalue is necessarily $-1$).  
By Lemma~\ref{l:dimension}, $g$ is a reflection if and only if $\dim P_g = d - 1$.
%Let $\Phi_1$ (respectively $\Phi_2$) denote the set of reflections such that the affine span of $P_g$ contains (respectively doesn't contain) a lattice point. Equivalently, a reflection $g$ lies in $\Phi_1$ if and only if $g$ fixes a lattice point in $M \times 1$. 
Let $s(P)$ denote the (normalized) surface area of $P$ (see Section~\ref{s:basicEhrhart}). 

\begin{corollary}\label{c:second}
With the notation above, the second leading coefficient $L_{d - 1}(m)$ is periodic in $m$ with period dividing $2$, and is given by the (virtual) character
\[ 
L_{d - 1}(m)(g) =  \left\{\begin{array}{cl} 
\frac{s(P)}{2} & \text{if } g = 1 %;
 \\ \vol{P_g} & \text{if } g \textrm{ is a reflection and } m \equiv 0  \mod \ind(P_g)
 % \in \Phi_1 \textrm{ or }  g \in \Phi_2 \textrm{ and } m \equiv 0 \mod 2
 % \\ \vol{P_g} & \text{if } g \in \Phi_2 \textrm{ and } m \equiv 0 \mod 2
   \\ 0 & \text{otherwise}.  
 \end{array}\right. 
\]
In particular, $L_{d - 1}(m)$ is independent of $m$ if and only if every reflection fixes a point in $M \times 1$. 
\end{corollary}
\begin{proof}
By Lemma~\ref{l:character} and the previous discussion, $\chi_{mP}(g) = f_{P_g}(m)$ is a quasi-polynomial of degree strictly less than $d - 1$ unless $g = 1$ or $g$ is a reflection. In the latter case, $g^2 = 1$ and Lemma~\ref{l:vertices} implies that $2P_g$ is a lattice polytope. 
Hence, $\ind(P_g) \in \{ 1, 2 \}$, and 
the first two statements now follow from basic properties of Ehrhart quasi-polynomials (see \eqref{e:dindex} in Section~\ref{s:basicEhrhart} and the surrounding discussion).  
The final statement follows from the observation that if $g$ is a reflection, then $\ind(P_g) = 1$ if and only if 
$g$ fixes a point in $M \times 1$. 
\end{proof}

\begin{remark}\label{r:second}
Recall from the introduction that the generating series of $L(m)$ can be written in the form 
\[
\sum_{m \ge 0} L(m) t^m = \frac{ \phi[t]}
%{ (-1)^{d + 1}(t  - 1)(t^d - V \, t^{d - 1} + \bigwedge^2 V \,  t^{d - 2}  - \cdots + (-1)^d \bigwedge^d V)}
%\bigwedge^d V - \bigwedge^{d - 1} V t + \cdots + (-1)^{d - 1} V t^{d - 1}+ (-1)^d )}
{(1- t)\det(I - \rho t)},
\]
for some power series $\phi[t] \in R(G)[[t]]$. 
It follows from Corollary~\ref{c:second} and Lemma~\ref{l:bad} below that if $\phi[t]$ is a polynomial, then  $L_{d - 1} = L_{d - 1}(m)$ is independent of $m$. 
%We remark that the latter fact could also be deduced from Theorem 4 in \cite{BSWMaximal} 
\end{remark}

%\comment{later: need to prove that if $\phi[t]$ is a polynomial, then $L_{d - 1}(m)$ is constant. See also Theorem 4 in Beck et al}

%%%%%%%%%%%%%%%%%%%%%%%%%%%%%%%%%%%%%%%%%%%%%%%
%%%%%%%%%%%%%%%%%%%%%%%%%%%%%%%%%%%%%%%%%%%%%%%
%%%%%%%%%%%%%%%%%%%%%%%%%%%%%%%%%%%%%%%%%%%%%%%
%%%%%%%%%%%%%%%%%%%%%%%%%%%%%%%%%%%%%%%%%%%%%%%

\excise{
\begin{remark}
 For any positive integer $m$, one may use Serre duality to show that we have isomorphisms of representations
 \[
%In our case, Serre duality implies that,, 
(-1)^d \chi(Y, L^{\otimes - m}) =
H^d(Y, L^{\otimes - m}) \cong H^0(Y, \mathcal{O}(K_Y) \otimes  L^{\otimes  m}) \otimes \bigwedge^d M_\C, %= L^*(mP) \otimes \bigwedge^d M_\C,
\] 
where we may identify $H^0(Y, \mathcal{O}(K_Y) \otimes  L^{\otimes  m})$ with $L^*(mP)$.
Since  $\chi(Y, L^{\otimes m}) = L(mP)$ for each non-negative integer $m$,  Theorem~\ref{t:combinatorics} implies that $\chi(Y, L^{\otimes m})$ is a quasi-polynomial in $m$ (for all integers $m$) with period dividing the exponent of $G$. 
 
 Careful: need $G$ acts trivially on $H^d(Y,  \mathcal{O}(K_Y) )$. 
 Is this a general fact?
%How is this compatible with Serre duality? Is the problem that $\chi(X, L^{\otimes m})$ is not a quasi-polynomial for all integer $m$? (only for non-negative). What is going on?
\end{remark}
}

%%%%%%%%%%%%%%%%%%%%%%%%%%%%%%%%%%%%%%%%%%%%%%%
%%%%%%%%%%%%%%%%%%%%%%%%%%%%%%%%%%%%%%%%%%%%%%%
%%%%%%%%%%%%%%%%%%%%%%%%%%%%%%%%%%%%%%%%%%%%%%%
%%%%%%%%%%%%%%%%%%%%%%%%%%%%%%%%%%%%%%%%%%%%%%%
\section{An equivariant analogue of the $h^*$-polynomial}\label{s:hstar}
%%%%%%%%%%%%%%%%%%%%%%%%%%%%%%%%%%%%%%%%%%%%%%%
%%%%%%%%%%%%%%%%%%%%%%%%%%%%%%%%%%%%%%%%%%%%%%%
%%%%%%%%%%%%%%%%%%%%%%%%%%%%%%%%%%%%%%%%%%%%%%%
%%%%%%%%%%%%%%%%%%%%%%%%%%%%%%%%%%%%%%%%%%%%%%%

The goal of this section is to study the power series $\phi[t]$ of virtual representations  introduced in the introduction, that  may be viewed as an 
equivariant analogue of the $h^*$-polynomial of a lattice polytope.

%We prove the main results of this paper in this section, and establish 

We will continue with the notation of Section~\ref{s2} and Section~\ref{s:Ehrhart}. 
That is, $G$ acts linearly on the  lattice $M' = M \oplus \Z$ of rank $d + 1$, and $P \subseteq M_\R \times  1$ is a $G$-invariant, $d$-dimensional lattice polytope. 
%With a slight abuse of notation, 
%for any positive integer $m$, 
%$\chi_{mP}$ (respectively $\chi^*_{mP}$)
%denotes 
%the complex permutation representation induced by the action of $G$ on $mP \cap M$ (respectively $\Int (mP) \cap M$), and
% $\chi_{mP}$ denotes the trivial representation when $m = 0$.  %Recall from the introduction that 
% Recall from the introduction that 
 If $R(G)$ denotes the %complex 
 representation ring of $G$ and $\rho: G \rightarrow GL(M)$, then we may 
write
%consider the following equality of power series in $R(G)[[t]]$
\begin{equation*}\label{e:hstar}
\sum_{m \ge 0} \chi_{mP} t^m = \frac{ \phi[t]}
%{ (-1)^{d + 1}(t  - 1)(t^d - V \, t^{d - 1} + \bigwedge^2 V \,  t^{d - 2}  - \cdots + (-1)^d \bigwedge^d V)}
%\bigwedge^d V - \bigwedge^{d - 1} V t + \cdots + (-1)^{d - 1} V t^{d - 1}+ (-1)^d )}
{(1- t)\det(I - \rho t)}, 
\end{equation*}
for some power series  $\phi[t] = \phi_{P,G}[t] = \sum_{i \ge 0} \phi_i t^i \in R(G)[[t]]$, where, by Lemma~\ref{l:exterior},  
%and virtual representations $\{  \phi_i^* \in R(G) \mid i \ge 0 \}$. 
%Here $\det(I - \rho t)$ is defined as follows: 
%if $M$ is the lattice defined by a translate of the lattice points in the affine span of $P$ to the origin, then $M$ inherits a linear $G$-action (see Section~\ref{s2}), and %Setting  $M_\C = M \otimes_\R \C$, we define 
% and we may consider the complex $G$-representation $M_\C = M \otimes_\R \C$. We set 
\[
\det(I - \rho t) = 1 - M_\C t + \wedge^2 M_\C t^2 - \cdots + (-1)^d \wedge^d M_\C t^d.  
\]
%where  $M_\C = M \otimes_\Z \C$. 
%Observe that $\phi_0$ is the trivial representation, and $\phi_1 = \chi_P - M_\C$. We will prove below that $\phi_1$ is an effective representation (Corollary~\ref{c:basic}). 

We first give an explicit description of $\phi[t]$ when $P$ is a simplex. 
% generalizing the well-known description of the $h^*$-polynomial of a simplex.
 Recall that $P$ is a \emph{simplex} if it has precisely $d + 1$ vertices $\{ v_0, \ldots, v_d \}$. In this case, %let $C$ denote the cone over $P$ and
  we define 
 \[
 \BOX(P) = \{ v \in M' \mid  v = \sum_{i = 0}^d a_i v_i \textrm{ for some } 0 \le a_i < 1 \}, 
 \]
 and let $\BOX(P)^g$ denote the elements of $\BOX(P)$ fixed by $g \in G$. 
Let $u: M' = M \oplus \Z \rightarrow \Z$ denote projection onto the second coordinate. 
%The result below is  well-known when $G$ is trivial. 
%The result below implies that the virtual representations $\phi_i$ are in fact effective representations when $P$ is a simplex. 

\begin{proposition}\label{p:simplex}
With the notation above, if $P$ is a simplex, then $\phi_i$ is the permutation representation induced by the action of $G$ on $\{ v \in  \BOX(P)  \mid u(v) = i \}$. In particular, 
%if $g$ acts on $M$ via an \comment{check} integer-valued matrix $A$, then
\[
\sum_{m \ge 0} f_{P_g}(m) t^m = \frac{\sum_{v \in \BOX(P)^g}   t^{u(v)} }{(1 - t)\det(I - \rho(g)t)},
\]
and the multiplicity of the trivial representation (respectively $\det(\rho)$) in $\phi_i$ equals the number of $G$-orbits of 
$\{ v \in  \BOX(P)  \mid u(v) = i \}$ (respectively  the number of $G$-orbits of 
$\{ v \in  \BOX(P)  \mid u(v) = i \}$ whose isotropy subgroup is contained in $\{ g \in G \mid \det(\rho(g)) = 1 \}$).
%If $h^*_i > 0$, then the trivial representation  occurs with non-zero multiplicity in $\phi_i$. 
% we have the following formula for the generating series of the Ehrhart quasi-polynomial of $P_g$ 
\end{proposition}
\begin{proof}
Since the vertices $\{ v_0, \ldots, v_d \}$ of $P$ form a basis of $M'_\C$, we may identify  the representation $M'_\C$ with the permutation representation induced by the action of $G$ on 
$\{ v_0, \ldots, v_d \}$. Hence we may identify $\Sym^\bullet M'_\C$ with the graded permutation representation of $G$ on $\{ \sum_{i = 0}^d b_i v_i \mid  b_i \in \Z_{\ge 0} \}$, where the latter set is graded by projection $M' = M \oplus \Z \rightarrow \Z$ onto the second coordinate. Similarly, let 
$V$ denote the graded permutation representation induced by the action of $G$ on $\BOX(P)$.  
%The tensor product $\Sym^\bullet M'_\C \otimes V$ is the graded permutation representation induced by the action of $G$ on $\{ (w, \sum_{i = 0}^d b_i v_i) \mid w \in \BOX(P),  b_i \in \Z_{\ge 0} \}$. 
If $C$ denotes the cone over $P$, then every lattice point $v$ in $C$ can be uniquely written as a sum 
$v = w + \sum_{i = 0}^d b_i v_i$, for some $w \in \BOX(P)$ and  $b_i \in \Z_{\ge 0}$. It follows that 
the tensor product $\Sym^\bullet M'_\C \otimes V$ is the graded permutation representation induced by the action of $G$ on $C \cap M'$. 
The first statement now follows since Lemma~\ref{l:exterior} implies that 
$\sum_{m \ge 0} \Sym^m M'_\C \, t^m = \frac{1}{(1 - t)\det(I - \rho t)}$. 
The second statement follows immediately by evaluating characters at $g \in G$ using Lemma~\ref{l:character}, and the final statement follows from Lemma~\ref{l:permutation}. 

%The final statement follows from the observation that 
%the trivial representation  occurs with non-zero multiplicity in a permuation representation. 
\end{proof}

The following remark can be useful for producing examples. 

\begin{remark}\label{r:pyramid}
The \define{pyramid} $\Pyr(P)$ of $P$ is the convex hull in $M'_\R$ of $P$ and the origin. One may verify that 
$\Pyr(P)$ is a $(d + 1)$-dimensional, $G$-invariant lattice polytope with $\phi_P[t] = \phi_{\Pyr(P)}[t]$. 
\end{remark} 

\begin{lemma}\label{l:rational}
For any $g$ in $G$, $\phi[t](g)$ is a rational function in $t$ that is regular at $t = 1$. 
\end{lemma}
\begin{proof}
By Lemma~\ref{l:character}, $\sum_{m \ge 0} \chi_{mP}(g) t^m = \sum_{m \ge 0} f_{P_g}(m) t^m$, and the latter generating series is a rational function with a pole of order at most $\dim P_g + 1$ at $t = 1$
\cite[Theorem~4]{BSWMaximal}. Hence Lemma~\ref{l:dimension} implies that
$(1 - t)\det(I - \rho(g) t) \sum_{m \ge 0} \chi_{mP}(g) t^m$ is regular at $t = 1$. 
\end{proof}

It follows from the lemma above that we may consider the rational class function
$\phi[1]$. Recall
that $M^g$ denotes the subspace of $M_\R$ fixed by $g$, and let  $(M^g)^\perp$ denote the orthogonal subspace in $M_\R$. Recall that the \emph{index} $\ind(Q)$ of a rational polytope $Q$ is the smallest positive integer $m$ such that the affine span of $m Q$ contains a lattice point (see Section~\ref{s:basicEhrhart}).

\begin{proposition}\label{p:rational}
With the notation above,
\[
%%\tilde{\phi}(g) = 
\phi[1](g)
=   \frac{ \dim(P_g)! \vol(P_g) \det(I - \rho(g))_{(M^g)^\perp }}{ \ind(P_g) }.
\]
In particular, $\phi[1]$ takes non-negative values. 
\end{proposition}
\begin{proof}
It follows from Lemma~\ref{l:character} and Section~\ref{s:basicEhrhart}
%Recall from Section~\ref{s:basicEhrhart} 
that 
if $N = \dim P_g$ and $r = \ind(P_g)$, then %Lemma~\ref{l:character} implies that
\[
\chi_{mP}(g) = f_{P_g}(m) = c_N(m)m^N + c_{N - 1}(m) m^{N - 1} +  \cdots + c_0(m), 
\]
where $c_i(m)$ is a periodic function in $m$, 
 and
\begin{equation*}
c_N(m) =   \left\{\begin{array}{cl} 
\vol(P_g) & \text{if } r | m %;
 \\ 0 & \text{otherwise}. \end{array}\right. 
\end{equation*}
It follows from \cite[Theorem~4]{BSWMaximal} that $\sum_{m \ge 0} c_i(m) m^i t^m$ is a rational function in $t$ with a pole at $t = 1$ of order at most $N$ for $i < N$. Also, 
\[
\sum_{m \ge 0} c_N(m)m^N t^m = \vol(P_g) \sum_{r | m} m^N t^m = r^N \vol(P_g)  \sum_{m \ge 0} m^N (t^r)^m. 
\] 
Here $\sum_{m \ge 0} m^N t^m = \frac{tA(N; t)}{(1 - t)^{N + 1}}$, where $tA(N; t)$ is the 
$N^\mth$ Eulerian polynomial,  and $A(N;1) = N!$ (cf. Section~\ref{s:hypercube}). 
%By Lemma~\ref{l:dimension},
It follows from Lemma~\ref{l:dimension} that $\phi[1](g)$ equals 
\[
[(1 - t)^{N + 1} \det(I - \rho(g)t)_{(M^g)^\perp} \sum_{m \ge 0} \chi_{mP}(g) t^m]|_{t = 1} = 
\]
\[
[(1 - t)^{N + 1} \det(I - \rho(g)t)_{(M^g)^\perp}  r^N \vol(P_g) \frac{t^rA(N; t^r)}{(1 - t^r)^{N + 1}}   ]|_{t = 1},
\]
and the first statement follows. The second statement follows since the complex eigenvalues of $\rho(g)$ come in conjugate pairs $\{ e^{\pm i \alpha} \mid \alpha \in \R \}$, and $(1 - e^{i \alpha} )(1 - e^{-i \alpha} ) =  2 - 2 \cos \alpha \ge 0$.  
\end{proof}

\begin{remark}\label{r:rational}
Observe that if $\phi[t]$ is a polynomial, then $\phi[1]$ is an integral-valued virtual character,  and the right hand side of Proposition~\ref{p:rational} is a (non-negative) integer. We conjecture that this holds in general (Conjecture~\ref{c:integer}). 
\end{remark}

%Finally, we wish to interpret this result in terms of the virtual representations $H^*_i$ of the introduction. 
The result below is a consequence of the  equivariant generalization of Ehrhart reciprocity in Theorem~\ref{t:combinatorics}. Recall from Section~\ref{s:basicEhrhart} that  the \emph{codegree} of $P$ is   $l = \min\{ m \in \Z_{> 0} \mid \Int(mP) \cap M \ne \emptyset \}$ and  the \emph{degree} $s = d + 1 - l$ of $P$ is  the degree of  $h^*(t)$. 

\begin{corollary}\label{c:reciprocity}
%If $\phi[t]$ is a polynomial, then 
With the notation above, 
\begin{equation*}%\label{e:interior}
\sum_{m \ge 1} \chi^*_{mP} t^m = \frac{ t^{d + 1}\phi[t^{-1}]}
%{ (-1)^{d + 1}(t  - 1)(t^d - V \, t^{d - 1} + \bigwedge^2 V \,  t^{d - 2}  - \cdots + (-1)^d \bigwedge^d V)}
%\bigwedge^d V - \bigwedge^{d - 1} V t + \cdots + (-1)^{d - 1} V t^{d - 1}+ (-1)^d )}
{(1- t)\det(I - \rho t)}.
\end{equation*}
In particular, if $\phi[t]$ is a polynomial then the degree of $\phi[t]$ is equal to the degree of $P$, and $\phi_s = \chi^*_{lP}$. 
 %Moreover, the multiplicity of the trivial representation in $\phi_s$ is at least the number of $G$-orbits of 
%$\Int(lP) \cap M$.
In this case, the multiplicity of the trivial representation (respectively $\det(\rho)$) in $\phi_s$ equals the number of $G$-orbits of 
$\Int(lP) \cap M$ (respectively  the number of $G$-orbits of 
$\Int(lP) \cap M$ whose isotropy subgroup is contained in $\{ g \in G \mid \det(\rho(g)) = 1 \}$).

% the 
 %trivial representation  occurs with non-zero multiplicity in 
\end{corollary}
\begin{proof}
%It will be enough to prove the equality at the level of characters.
%% Fix $g \in G$ acting on $M$ via an integer matrix $A$, and let $\phi[t](g)$ denote the evaluation of the character of  $\phi[t]$ at $g$. 
 By Lemma~\ref{l:character}, $\chi^*_{mP}(g) = f_{P_g}^\circ(m)$ for any positive  integer $m$. 
%By Lemma~\ref{l:character}, we need to prove that
%\[
%\sum_{m \ge 1} f_{P_g}^\circ (m) t^m = \frac{ t^{d + 1}H^*(g;t^{-1})}
%{(1 - t)\det(I - tA)}. 
%\]
%By Theorem~\ref{t:combinatorics}, 
By Ehrhart Reciprocity (see Section~\ref{s:basicEhrhart}),  $f_{P_g}^\circ(m) =  (-1)^{\dim P_g} f_{P_g}(-m)$ for any positive integer $m$, and, by Proposition~4.2.3 in \cite{StaEnumerative}, 
$\sum_{m \ge 1} f_{P_g}(-m) = \frac{-\phi[t^{-1}](g)}{(1 - t^{-1})\det(I - \rho(g)t^{-1})}$. Hence
\[
\sum_{m \ge 1} f_{P_g}^\circ (m) t^m = \frac{(-1)^{\dim P_g - 1}\phi[t^{-1}](g)}{(1  - t^{-1})\det(I - \rho(g)t^{-1})}
 = \frac{(-1)^{d - \dim P_g}t^{d + 1}\phi[t^{-1}](g)}{\det \rho(g) (1 - t)\det(I - \rho(g)^{-1}t)}. 
\]
Since the eigenvalues of $\rho(g)$ are roots of unity and $\rho(g)$ is integer-valued, $\det(I - t \rho(g)^{-1}) = \det(I - t \,  \bar{\rho(g)}) = \det(I - t \rho(g))$. Moreover, 
by Lemma~\ref{l:determinant}, $\det \rho(g) = (-1)^{d - \dim P_g}$. We conclude that 
\[
\sum_{m \ge 1} f_{P_g}^\circ (m) t^m = \frac{ t^{d + 1}\phi[t^{-1}](g)}
{(1 - t)\det(I - \rho(g)t)}, 
\]
as desired. The second statement is immediate, and the final statement follows from Lemma~\ref{l:permutation}.  
%The final statement follows from the observation that 
%the trivial representation  occurs with non-zero multiplicity in a permuation representation. 
%the latter expression is equal to  
%\[
% = \frac{t^{d + 1}H^*(g; t^{-1})}{(1 - t)\det(I - t A^{-1})} = \det A' \frac{(-1)^{d - \dim P_g}t^{d + 1}H^*(g; t^{-1})}{\det(I - t(A')^{-1})}
%\]
\end{proof}

If $V$ and $W$ are virtual representations of $G$, then we write $V \ge W$ if $V - W$ is an effective representation. %\define{not so well-defined}
%Let $M^G$ denote the subspace of $M_\R$ fixed by $G$. 

\begin{corollary}\label{c:basic}
With the notation above, $\phi_0 = 1$ and $\phi_1$ is an effective representation. 
%The multiplicity of the trivial representation in $\phi_1$ equals the number of $G$-orbits of $P \cap M - \dim M^G - 1$.  
Moreover, if $\phi[t]$ is a polynomial, then $\phi_1 \ge \phi_d \ge 0$. 
% and the trivial representation  occurs with non-zero multiplicity in $\phi_0$, $\phi_1$ and $\phi_d$.
\end{corollary}
\begin{proof}
Since $\chi_{mP}$ is the trivial representation when $m = 0$, it follows from the definitions that $\phi_0 = 1$ and $\phi_1 = \chi_P - M'_\C$.   
%By Corollary~\ref{c:reciprocity},  $H_d^* = L^*(P)$ is the permutation representation of $G$ on the interior lattice points of $P$.
 Let $W$ denote the permutation representation of $G$ acting on the vertices $\{ v_i \mid i \in I  \}$ of 
 $P$, %\subseteq M_\R \times \{ 1 \} \subseteq M'_\R$, 
and let $W' \subseteq W$ be the $G$-submodule consisting of all relations satisfied by the vectors $\{ v_i \mid i \in I  \}$ in $M'_\R$. Since $P$ is $d$-dimensional, the vertices of $P$ span $M'_\R$ as a vector space, and we have an isomorphism of $G$-representations $W/W' \cong M'_\C$. Since 
$W \oplus \chi^*_P$ is a subrepresentation of $\chi_P$ by definition, we conclude that 
$\chi_P - M'_\C - \chi^*_P$ is an effective representation. Finally, Corollary~\ref{c:reciprocity} implies that if $\phi[t]$ is a polynomial, then $\phi_d = \chi^*_P$, and the result follows. 
\end{proof}

The example below demonstrates that we cannot hope that $\phi_i$ is effective for $i \ge 2$, 
%when $\phi[t]$ is not a polynomial, 
even if $h_i^* > 0$. 

\begin{example}
Let $G = \Z/2\Z$ with generator $\tau$, and let $\chi: G \rightarrow \C$ be the linear character sending $\tau$ to $-1$, so that the irreducible representations of $G$ are $\{ 1, \chi \}$. Let $P$ be the convex hull of $(0,0)$, $(3,0)$, $(0,3)$ and $(3,3)$, and consider the action of $G$ on $M = \Z^2$ given by
%such that $\sigma$  acts via the matrix 
\begin{displaymath}
\tau \mapsto \left(\begin{array}{c c}
-1 & 0 \\ 
0 & 1 
\end{array}\right). 
\end{displaymath}
Then, with the notation of Section~\ref{s2}, $P$ is $G$-invariant `up to translation', and one computes that $\phi_2 = 5 - \chi$. % and $h_2^* = 4$. 
Observe that in this example $\phi[t]$ is not a polynomial by Lemma~\ref{l:bad}. 
\end{example}

\excise{
We record the following basic corollary as it may be useful in computations. 

\begin{corollary}
With the notation above, if $\phi[t]$ is a polynomial, then the trivial representation  occurs with non-zero multiplicity in $\phi_s$. If we further assume that $P$ contains an interior lattice point, then  the trivial representation  occurs with non-zero multiplicity in $\phi_1$. 
\end{corollary}
\begin{proof}
This follows immediately from Corollary~\ref{c:reciprocity} and Corollary~\ref{c:basic}, using the observation that 
the trivial representation  occurs with non-zero multiplicity in a permuation representation. 
\end{proof}
}

Recall that a $d$-dimensional lattice polytope $P$ in $M$ is \define{reflexive} if the origin is the unique interior lattice point of $P$ and every non-zero lattice point in $M$ lies in the boundary of $mP$ for some positive integer $m$. We have the following equivariant version of Theorem~\ref{t:basicreflexive}. 
Recall that the degree $s$ of $P$ is the degree of $h^*(t)$ and the codegree of $P$ is $l = d + 1 - s$.

%It is a result of Hibi  \cite{HibDual} that $P$ is a translate of a reflexive polytope if and only if $h^*(t) = t^d h^*(t^{-1})$. \comment{might need to recall notation} \comment{recall Stanley's result, including $f_{P}^\circ(m) = f_{P}(m - l)$ for $m \ge l$} \comment{Want notation with $P$ at height $1$ and $\sigma$ equal to the cone over $P$} 
%Observe that if a reflexive polytope $P$ is $G$-invariant `up to translation', then $P$ is $G$-invariant, since 

\begin{corollary}\label{c:reflexive}
With the notation above, if $P$ is a $G$-invariant lattice polytope of degree $s$ and codegree $l$, then the following are equivalent
\begin{itemize}
\item $\phi[t] = t^s \phi[t^{-1}]$,
\item  $\chi^*_{mP} = \chi_{(m - l)P}$ for $m \ge l$,
%\item  $h^*_{P}(t) = t^{s} h^*_{P}(t^{-1})$,
\item  $f^\circ_P(m) = f_P(m - l)$ for $m \ge l$,
\item  $h^*_{P}(t) = t^{s} h^*_{P}(t^{-1})$,
\item  $lP$ is a translate of a reflexive polytope. 
\end{itemize}
%Moreover, if $\phi[t]$ is a polynomial, then the latter conditions are also equivalent to $\phi[t] = t^s \phi[t^{-1}]$. 
\end{corollary}
\begin{proof}
The third, fourth and fifth conditions are equivalent by Theorem~\ref{t:basicreflexive}, and the second condition clearly implies the third. The fact that the first two conditions are equivalent is a formal consequence of Corollary~\ref{c:reciprocity}. 
If $lP$ is a translate of a reflexive polytope, then let $v$ denote the unique (and hence $G$-invariant) interior lattice point in $lP$. For $m \ge l$,  
the equality $f_{P}^\circ(m) = f_{P}(m - l)$ implies that  $\Int(mP) \cap M' =  (m - l)P \cap M' + v$, and hence $\chi^*_{mP} = \chi_{(m - l)P}$. 
%Finally, if $\phi[t]$ is a polynomial, then the fact that the first condition is equivalent to $\phi[t] = t^s \phi[t^{-1}]$ is a formal consequence of Corollary~\ref{c:reciprocity}. 

\excise{
If $\chi^*_{mP} = \chi_{(m - l)P}$ for $m \ge l$, then $f_{P}^\circ(m) = f_{P}(m - l)$ for $m \ge l$, and 
 $lP$ is a translate of a reflexive polytope by Theorem~\ref{t:basicreflexive}. 

If $H^*(t) = t^s H^*(t^{-1})$, then $h^*(t) = t^s h^*(t^{-1})$, and, by Stanley's result above, $lP$ is a translate of a reflexive polytope. Conversely, suppose that $lP$ is a translate of a reflexive polytope. %Observe that, 
We have seen above that $f_{P}^\circ(m) = f_{P}(m - l)$ 
for any $m \ge l$. More specifically, if %$\sigma$ denotes the cone over $P \subseteq M_\R \times \{1\}$ and  
$v$ denotes the unique interior lattice point in $lP$, then % \subseteq M_\R \times \{l\}$, then 
%the interior lattice points in $mP$ %= \sigma \cap (M_\R \times \{m \})$ 
$\Int(mP) \cap M' =  (m - l)P \cap M' + v$.
%are all of the form $(m - l)P + v$. 
Since $v$ is $G$-invariant, we conclude that 
$L^*(mP) = L((m - l)P)$ for $m \ge l$. 
Hence
$\sum_{m \ge 1} L^*(mP)t^m = t^l \sum_{m \ge 0} L(mP)t^m$, and 
Corollary~\ref{c:reciprocity} implies that $t^{d + 1} H^*(t^{-1}) = t^{l} H^*(t)$ as desired.
%the result %now 
%follows %immediately 
%from 
%Corollary~\ref{c:reciprocity}. 
}
\end{proof}

\begin{remark}
Example~\ref{e:badexample2} demonstrates that $\phi[t]$ is not necessarily a polynomial when $P$ is reflexive. 
\end{remark}

It is clear from the definitions that $\phi_i(g) \in \Z$ for all $g \in G$. Moreover, if $\phi[t]$ is a polynomial, then Corollary~\ref{c:reciprocity} implies that the leading term $\phi_s$ is a permutation representation, and hence $\phi_s(g)$ is a non-negative integer for all $g \in G$. The example below demonstrates that one can not expect that $\phi_i(g) \in \Z_{\ge 0}$ in general.

\begin{example}\label{e:counter}
Let $G = \Z/6\Z$ with generator $\sigma$, and let $\chi: G \rightarrow \C$ be the linear character sending $\sigma$ to $\zeta_6 = e^{\frac{2 \pi i}{6}}$. Then the irreducible characters of $G$ are $\{ 1, \chi, 
 \ldots, \chi^5 \}$. %\chi^2, \chi^3, \chi^4
Consider the representation of $G$ on $M = \Z^2$ via
\begin{displaymath}
\sigma \mapsto A = \left(\begin{array}{c c}
0 & 1 \\ 
-1 & 1 
\end{array}\right),
\end{displaymath}
and let $P$ be the (reflexive) lattice polytope with vertices $\{ \pm (1, 0), \pm (0, 1), \pm (1,1) \}$. 
One computes that $\phi[t] = 1 + (1 + \chi^2 + \chi^3 + \chi^4)t + t^2$, and 
$\phi[t](\sigma) = 1 - t + t^2$.
%the evaluation of the character of $\phi[t]$ at $\sigma$ is  $\phi(\sigma; t) = 1 - t + t^2$. 
In fact, $P$ is a \emph{non-singular} reflexive polytope in the sense that the vertices of each facet of $P$ form a basis of $M$, and the fact that $\phi[t]$ is a polynomial is guaranteed by  Corollary~\ref{c:dim2} (and Proposition~\ref{p:cohomology}). 

Observe that the character $\phi[1] = 3  + \chi^2 + \chi^3 + \chi^4$  has non-negative values. In fact, 
$\phi[1]$ is isomorphic to a permutation representation with isotropy subgroups $\{ 1, \Z/2\Z, \Z/3\Z \}$. 
\end{example}

If $H$ is a finite group acting on a lattice $N$ via $\rho': H \rightarrow GL(N)$, and $Q$ is an $H$-invariant lattice polytope, then 
the direct product $P \times Q = \{ (p,	q) \mid m \in P, n \in Q \}$ and the direct sum $P \oplus Q = \conv \{ P \times 0, 0 \times Q \}$ are $(G \times H)$-invariant lattice polytopes. We may and will regard a $G$-representation (respectively an $H$-representation) as a $(G \times H)$-representation via the projection of $G \times H$ onto $G$ (respectively $H$). 

\begin{proposition}\label{p:Braun}
With the notation above,  $\chi_{m(P \times Q)} = \chi_{mP} \cdot \chi_{mQ}$, and if $P$ is a reflexive polytope and $Q$ contains the origin in its interior, then $\phi_{P \oplus Q}[t] = \phi_P[t] \cdot \phi_Q[t]$. 
\end{proposition}
\begin{proof}
The first statement is clear since the lattice points of $m(P \times Q)$ are $\{ (p,q) \mid p \in mP \cap M, q \in mQ \cap N\}$. For the second statement, it follows from Braun's proof of \cite[Theorem~1]{BraEhrhart}
that the lattice points of $m(P \oplus Q)$ are $\{ (p,q) \mid p \in \partial(kP) \cap M, q \in (m - k)Q \cap N, 
0 \le k \le m \}$, and hence $\chi_{m(P \oplus Q)} = \chi_{mQ} + \sum_{k = 1}^m (\chi_{kP} - \chi_{(k - 1)P})
\cdot \chi_{(m - k)Q}$. We compute
\[
 \frac{\phi_{P \oplus Q}[t] }{(1 - t)\det(I - (\rho + \rho')t)} = 
\sum_{m \ge 0} \chi_{m(P \times Q)}t^m = 
\]
\[
\sum_{m \ge 0} [\chi_{mQ} + \sum_{k = 1}^m (\chi_{kP} - \chi_{(k - 1)P})
\cdot \chi_{(m - k)Q}] t^m = 
(1 - t) \sum_{m \ge 0} \chi_{mP}t^m \sum_{m \ge 0} \chi_{mQ}t^m = 
\]
\[
(1 - t) \frac{ \phi_P[t]}{(1 - t)\det(I - \rho t)}   \frac{\phi_Q[t] }{(1 - t)\det(I - \rho' t)}
= \frac{\phi_P[t] \cdot \phi_Q[t]}{(1 - t)\det(I - (\rho + \rho')t)}. 
\]

\end{proof}

%%%%%%%%%%%%%%%%%%%%%%%%%%%%%%%%%%%%%%%%%%%%%%%
%%%%%%%%%%%%%%%%%%%%%%%%%%%%%%%%%%%%%%%%%%%%%%%
%%%%%%%%%%%%%%%%%%%%%%%%%%%%%%%%%%%%%%%%%%%%%%%
%%%%%%%%%%%%%%%%%%%%%%%%%%%%%%%%%%%%%%%%%%%%%%%
\section{Effectiveness of representations}\label{s:effectiveness}
%%%%%%%%%%%%%%%%%%%%%%%%%%%%%%%%%%%%%%%%%%%%%%%
%%%%%%%%%%%%%%%%%%%%%%%%%%%%%%%%%%%%%%%%%%%%%%%
%%%%%%%%%%%%%%%%%%%%%%%%%%%%%%%%%%%%%%%%%%%%%%%
%%%%%%%%%%%%%%%%%%%%%%%%%%%%%%%%%%%%%%%%%%%%%%%

The goal of this section is to provide criterion to determine whether the power series $\phi[t]$ 
%considered in the previous section
is effective and whether it is a polynomial. We continue with the notation of Section~\ref{s2}, Section~\ref{s:Ehrhart} and Section~\ref{s:hstar}. 

Recall that the generating series of $\{ \chi_{mP} \}_{m \ge 0}$ can be written in the form 
\[
\sum_{m \ge 0} \chi_{mP} t^m = \frac{ \phi[t]}
%{ (-1)^{d + 1}(t  - 1)(t^d - V \, t^{d - 1} + \bigwedge^2 V \,  t^{d - 2}  - \cdots + (-1)^d \bigwedge^d V)}
%\bigwedge^d V - \bigwedge^{d - 1} V t + \cdots + (-1)^{d - 1} V t^{d - 1}+ (-1)^d )}
{(1- t)\det(I - \rho t)},
\]
for some power series $\phi[t] \in R(G)[[t]]$.
% where 
%if an element $g \in G$ acts on $M$ via an integer matrix $A$, then
%the \comment{fix} evaluation of the character of $\det(I - \rho t)$ at $g$ equals $\det(I - tA)$. 
We begin with a criterion that guarantees that $\phi[t]$ is a polynomial.

\begin{lemma}\label{l:polynomial}
If $P_g$ is a lattice polytope for all $g$ in $G$, then $\phi[t]$ is a polynomial. 
\end{lemma}
\begin{proof}
By Lemma~\ref{l:character} and Lemma~\ref{l:dimension}, 
\[
\sum_{m \ge 0} \chi_{mP}(g) t^m = \sum_{m \ge 0} f_{P_g}(m) t^m = \frac{h^*_{P_g}(t)}{(1 - t)^{\dim P_g + 1}}
= \frac{h^*_{P_g}(t)\det(I - \rho(g)t)_{(M^g)^\perp}}{(1 - t)\det(I - \rho(g)t)}, 
\]
where $\det(I - \rho(g)t) = (1 - t)^{\dim P_g}\det(I - \rho(g)t)_{(M^g)^\perp}$ (cf. Proposition~\ref{p:rational}), and hence $\phi[t](g) = h^*_{P_g}(t)\det(I - \rho(g)t)_{(M^g)^\perp}$. 
\end{proof}

\begin{remark}\label{r:multiple}
It follows from Lemma~\ref{l:vertices} and the lemma above that $\phi_{mP}[t]$ is a polynomial if the exponent of $G$ divides $m$. 
\end{remark}

We next provide a negative result. 
Recall that the \emph{index} of a rational polytope $Q$ is the smallest positive integer $m$ such that the affine span of $m Q$ contains a lattice point (see Section~\ref{s:basicEhrhart}).

\begin{lemma}\label{l:bad}
With the notation above, fix an element $g$ in $G$, and let $r$ denote the index of $P_g$. 
If $rP_g$ is a lattice polytope, and  $\dim P_g > \frac{d - r + 1}{r}$, then $\phi[t]$ is not a polynomial. 
In particular, if there exists a reflection that doesn't fix a point in $M \times 1$, then $\phi[t]$ is not a polynomial.  
\end{lemma}
\begin{proof}
Recall from Section~\ref{s:basicEhrhart} that $f_{P_g}(m) = 0$ unless $r | m$. 
Hence, by Lemma~\ref{l:dimension}, 
\[
\sum_{m \ge 0} f_{P_g}(m) t^m = \frac{h_{rP_g}^*(t^r)}{(1 - t^r)^{\dim P_g + 1}} = 
\frac{h_{rP_g}^*(t^r)\det(I - \rho(g)t)_{(M^g)^\perp}}{(1 - t)\det(I - \rho(g)t)(1 + t + \cdots + t^{r - 1})^{\dim P_g + 1}},
\]
where $\det(I - \rho(g)t) = (1 - t)^{\dim P_g}\det(I - \rho(g)t)_{(M^g)^\perp}$ (cf. Proposition~\ref{p:rational}). If $\zeta$ is an $r^\mth$ root of unity, then 
$\zeta$ is not a root of $h_{rP_g}^*(t^r)$ %(in fact, $h_{rP_g}^*(1)$ is the `normalized' volume of $rP_g$).  
(in fact, $h_{rP_g}^*(1)$ is equal to $(\dim P_g)!$ times the volume of $rP_g$). 
%since the coefficients of an $h^*$-polynomial are non-negative and hence $h_{rP_g}^*(1) > 0$. 
Hence if $\phi[t]$ is a polynomial, then $(1 + t + \cdots + t^{r - 1})^{\dim P_g + 1}$ divides $\det(I - \rho(g)t)_{(M^g)^\perp}$. However, 
$\det(I - \rho(g)t)_{(M^g)^\perp}$ has degree $d - \dim P_g < (r - 1)(\dim P_g + 1)$, a contradiction. 
%and, by Lemma~\ref{l:dimension}, $\det(I - tA) = (1 - t)^{\dim P_g} \phi[t]$, where $\Phi(1) \ne 0$.  
%the roots of $\phi[t]$ are $k^\mth$ roots of unity for various $k > 1$. 
%$h_{rP_g}^*(t^r)$ is the $h^*$-polynomial of $rP_g$. 

If $g^2 = 1$, then $2P_g$ is a lattice polytope by Remark~\ref{r:lattice} below. If $g$ is a reflection that doesn't fix a point in $M \times 1$, then $r = 2$, $\dim P_g = d - 1$, and the second statement follows. 
\end{proof}

\begin{remark}\label{r:lattice}
With the notation of the above lemma, if $g^r$ acts trivially on $M$, then $rP_g$ is a lattice polytope by  Lemma~\ref{l:vertices}. 
\end{remark}

\begin{example}\label{e:badexample}
Let $G = \Z/2\Z$ with generator $\tau$, and
% let $\chi: G \rightarrow \C$ be the linear character sending $\sigma$ to $\zeta_6 = e^{\frac{2 \pi i}{6}}$. Then the irreducible characters of $G$ are $\{ 1, \chi, 
 %\ldots, \chi^5 \}$. %\chi^2, \chi^3, \chi^4
consider the representation of $G$ on $M = \Z^3$ (cf. Example~\ref{e:breakdown}) via
\begin{displaymath}
\tau \mapsto A = \left(\begin{array}{c c c}
-1 & 0 &  1 \\ 
0 & 1  & 0 \\
0 & 0  & 1 
\end{array}\right). 
\end{displaymath}
If $P$ is the square with vertices $(0,0,1)$, $(1,0,1)$, $(0,1,1)$ and $(1,1,1)$, then $P$ is $G$-invariant.
Since $\tau$ is a reflection with no fixed points in $M \times 1$, Lemma~\ref{l:bad} implies that $\phi[t]$ is not a polynomial. Similarly, one can show that $\phi_{mP}[t]$ is not a polynomial when $m$ is odd. 
By taking the convex hull of $P$ and the origin, Remark~\ref{r:pyramid} implies that one obtains a lattice polytope with a $G$-fixed point such that $\phi[t]$ is not a polynomial. 

%For example, let $G = \left\langle \tau \right\rangle = \Z^2$ act on $M = \Z^3$ via 
%We refer the reader to Section~\ref{s:hypercube} (see, in particular, Example~\ref{e:breakdown}) for examples where the above lemma applies $\phi[t]$ is not a polynomial. One can use the pyramid construction in Remark~\ref{r:pyramid} to construct examples where $P$ has a fixed point and $\phi[t]$ is not a polynomial. 
\end{example}

\begin{example}\label{e:badexample2}
We present an example to show that $\phi[t]$ need not be a polynomial when $P$ contains a $G$-fixed interior lattice point. As in the previous example, let $G = \Z/2\Z$ with generator $\tau$, and
% let $\chi: G \rightarrow \C$ be the linear character sending $\sigma$ to $\zeta_6 = e^{\frac{2 \pi i}{6}}$. Then the irreducible characters of $G$ are $\{ 1, \chi, 
 %\ldots, \chi^5 \}$. %\chi^2, \chi^3, \chi^4
consider the representation of $G$ on $M = \Z^3$ via
\begin{displaymath}
\tau \mapsto A = \left(\begin{array}{c c c}
-1 & 0 &  1 \\ 
0 & 1  & 0 \\
0 & 0  & 1 
\end{array}\right). 
\end{displaymath}
Let $P$ be the $G$-invariant, $3$-dimensional lattice polytope with vertices $\pm(0,0,1)$, $\pm(1,0,1)$, $\pm(0,1,1)$ and $\pm(1,1,1)$, and observe that the origin is a $G$-fixed lattice point in the interior of $P$. In fact, $P$ is reflexive and has $h^*$-polynomial $h^*(t) = 1 + 5t + 5t^2 + t^3$ (cf. Corollary~\ref{c:reflexive}). If $M^\tau$ denotes the lattice fixed by $\tau$, then we have an isomorphism
$M^\tau \cong \Z^2, \; (1,0,2) \mapsto e_1, \; (0,1,0) \mapsto e_2$. Under this isomorphism, $P_\tau$ corresponds to the rational polytope with vertices $\pm \frac{e_1}{2} $ and $\pm  (\frac{e_1}{2} + e_2)$. 
One computes that $2P_\tau$ is a $2$-dimensional lattice polytope with $h^*_{2P_\tau}(t) = 1 + 6t + t^2$ and $f_{P_\tau}(m) = f_{2P_\tau}(\lfloor \frac{m}{2} \rfloor)$, and hence
\[
\sum_{m \ge 0} f_{P_\tau}(m) t^m = \frac{(1 + t)h_{2P_\tau}^*(t^2)}{(1 - t^2)^{3}} = 
\frac{1 + 6t^2 + t^4}{(1 - t)\det(I - tA)(1 + t)}. 
\]
We conclude that 
 if $\chi: G \rightarrow \C$ denotes the linear character sending $\tau$ to $-1$, then
 \[
 \phi[t] = \frac{1 + 3t + 8t^2 + 3t^3 + t^4}{1 + t} + \frac{t(3 + 2t + 3t^2)}{1 + t} \chi. 
 \]
%The lattice $M^\tau$  fixed by $\tau$ has rank $2$, and, after fixing an isomorphism 
%$M^\tau \cong \Z^2$, $(1,0,2) \mapsto e_1$, $(1,1,2) \mapsto e_2$. 
\end{example}

We say that $\phi[t]$ is \define{effective} if all the virtual representations $\phi_i$ are effective representations. Note that if $\phi[t]$ is effective, then $\phi[t]$ is a polynomial. For example, 
if $G = 1$, then $\phi[t] = h^*(t)$ is a polynomial with non-negative coefficients, and if $P$ is a simplex, then Proposition~\ref{p:simplex} implies that %$\phi[t]$ is effective. 
the representations $\phi_i$ are effective.  
For the remainder of the section we will provide criterion that guarantee that $\phi[t]$ is effective. 

%the representations $\phi_i$ are effective, and hence that $\phi[t]$ is a polynomial. 

We briefly recall some basic facts about toric varieties, and refer the reader to \cite{FulIntroduction} and \cite{TevCompactifications} for details. The lattice polytope $P$ determines a  complex, projective $d$-dimensional toric variety $Y = Y_P$ and an ample line bundle $L$ on $Y$. The toric variety $Y$ is a disjoint union of tori $T_Q$ of dimension $\dim Q$, indexed by the faces $Q$ of $P$. 
A section $s \in \Gamma(Y, L)$ determines a hypersurface $X = X(s)$ in $Y$, and we say that $X$ is \define{non-degenerate} if $X \cap T_Q$ is a smooth (possibly empty) hypersurface in $T_Q$ for all $Q \subseteq P$. Non-degenerate hypersurfaces were first studied by Khovanski{\u\i} \cite{KhoNewton}, %and \cite{KhoNewtonP}, 
and,
recently, have been extended to the notion of a \emph{Sch\"on} subvariety of a torus by Televev 
 \cite{TevCompactifications}.

Equivalently, a hypersurface  $X^\circ = \{ f = \sum_{ u \in P \cap M} a_u \chi^u = 0 \} \subseteq T = T_P$ 
is \emph{non-degenerate with Newton polytope $P$} if $\{ f_Q =  \sum_{ u \in Q \cap M} a_u \chi^u = 0 \} \subseteq T$ is a smooth (possibly empty) hypersurface in $T$ for all $Q \subseteq P$. One can show that these two notions of non-degenerate coincide. That is, $X = \bar{X^\circ}$ is non-degenerate if and only if $X^\circ = X \cap T$ is non-degenerate. The key point in proving this equivalence is the fact that 
\begin{equation}\label{e:nondeg}
\{ f_Q = 0 \} \cong (X \cap T_Q) \times (\C^*)^{d - \dim Q}. 
\end{equation}

Recall that $G$ acts on $M$ and leaves $P$ invariant `up to translation'. There is an induced action of 
$G$ on $Y$ via toric morphisms satisfying $g^* L \cong L$ for all $g \in G$. Let 
$\Gamma(Y, L)^G \subseteq \Gamma(Y, L)$ denote the sub-linear system of $G$-invariant sections of $L$. The following result is proved in \cite{YoRepresentations} using the theory of mixed Hodge structures. 

\begin{theorem}\cite{YoRepresentations}\label{t:nondeg}
With the notation above, if there exists a $G$-invariant non-degenerate hypersurface with Newton polytope $P$, then %the representations $\phi_i$ are effective. 
$\phi[t]$ is effective. 
In particular, this assumption holds if the linear system $\Gamma(Y, L)^G$ on $Y$ is base point free. 
\end{theorem}

\begin{remark}\label{r:bpf}
The second statement of this theorem is an easy application of Bertini's theorem (see Corollary~10.9 and Remark~10.9.2 in \cite{HarAlgebraic}). In order to compute examples, we give the following condition that, using \eqref{e:nondeg}, one can show is equivalent to the condition that $\Gamma(Y, L)^G$ is base point free: for each face $Q \subseteq P$, the linear system 
\begin{equation}\label{e:bpf}
\{ \sum_{u \in Q \cap M} a_u \chi^u \mid a_u = a_{u'} \textrm{ if } u, u' \textrm{ lie in the same $G_Q$-orbit} \} 
\end{equation}
 on $T$ is base point free, where $G_Q$ denotes the stabilizer of $Q$.
%subgroup of $G$ that fixes $Q$. 
%$u \sim u'$ if $u$ and $u'$  lie in the same $G_Q$-orbit of  $Q \cap M$. 
\end{remark}

\begin{remark}
Continuing with the notation of Remark~\ref{r:bpf} above, observe that
 in order to guarantee the existence of a $G$-invariant non-degenerate hypersurface, it is enough to check condition \eqref{e:bpf} holds for faces $Q \subseteq P$ with $\dim Q > 1$. Indeed, if $Q$ is a vertex, then the condition holds automatically since the non-zero elements of the linear system define the empty hypersurface in $T$. If $\dim Q = 1$, then we claim that an element (and hence a non-empty open subset) of the linear system 
  \eqref{e:bpf} defines a smooth hypersurface in $T$. This follows since $T_Q \cong \C^* = \Spec \C[t^{\pm 1}]$, and there exists an isomorphism $T \cong \C^* \times (\C^*)^{d - 1}$ such that \eqref{e:bpf} is isomorphic to a sub-linear system of $\{ \sum_{i = 0}^s a_i t^i = 0 \}$. The claim follows since the polynomial $\{ \sum_{i = 0}^s t^i = 0 \}$ has $s$ distinct roots. 
\end{remark}

The following corollary is immediate from the above two remarks. 
% and will be useful for applications. 

\begin{corollary}\label{c:fixed}
If every  face $Q$ of $P$ with $\dim Q > 1$ contains a lattice point %in its relative interior  
that is $G_Q$-fixed, where $G_Q$ denotes the stabilizer of $Q$,  then 
$\phi[t]$ is effective. 
%the representations $\phi_i$ are effective. 
\end{corollary}

\begin{remark}
The existence of a $G$-fixed lattice point is not necessary for $\phi[t]$ to be effective. For example, if $G = \Sym_{d + 1}$ acts on $\Z^{d + 1}$ via the standard representation and $P$ is the convex hull of the basis vectors 
$e_0, \ldots, e_d$, then $P$ has no $G$-fixed lattice points and $\phi[t] = 1$ by Proposition~\ref{p:simplex}. 
\end{remark}

We have the following two applications. 

\begin{corollary}\label{c:dim2}
If $\dim P = 2$ and $P$ contains a $G$-fixed lattice point, then $\phi[t]$ is effective. 
\end{corollary}
\begin{proof}
This follows immediately from Corollary~\ref{c:fixed}. 
\end{proof}

\begin{remark}
The above corollary is false if $\dim P > 2$ by Example~\ref{e:badexample} (and Example~\ref{e:badexample2}). 
\end{remark}

The corollary below should be compared with Remark~\ref{r:multiple}. 

\begin{corollary}\label{c:multiple}
If the order of $G$ divides $m$, then $\phi_{mP}[t]$ is effective. 
\end{corollary}
\begin{proof}
As in the proof of Lemma~\ref{l:vertices}, fix a face $Q$ of $P$ and let $\{ v_i \mid i \in \iota \}$ denote a $G_Q$-orbit of $Q \cap M$.  The rational point $v_\iota = \frac{1}{|\iota|} \sum_{i \in \iota} v_i$ 
is $G_Q$-fixed and $m v_\iota$ is a lattice point if  $|\iota|$ divides $m$. 
Finally note that $|\iota |$ divides $|G_Q|$ that divides  $|G|$, and apply  Corollary~\ref{c:fixed}. 
\end{proof}

%\begin{remark}
%It follows that $\phi_{mP}[t]$ is effective if the order of $G$ divides $m$. 
%Indeed, as in the proof of Lemma~\ref{l:vertices}, fix a face $Q$ of $P$ and let $\{ v_i \mid i \in \iota \}$ denote a $G_Q$-orbit of $Q \cap M$.  The rational point $v_\iota = \frac{1}{|\iota|} \sum_{i \in \iota} v_i$ 
%is $G_Q$-fixed and $m v_\iota$ is a lattice point if  $|\iota|$ divides $m$. 
%Finally note that $|\iota |$ divides $|G_Q|$ that divides  $|G|$. 

%In Example~\ref{e:badexample}, we see that $\phi_{mP}[t]$ is effective if and only if $m$ is an even, positive integer. 
%\end{remark}
\begin{example}
In Example~\ref{e:badexample}, $G = \Z/2\Z$, $\dim P = 2$, and we see that $\phi_{mP}[t]$ is effective if and only if $m$ is an even, positive integer. 
\end{example}

\begin{example}\label{e:proper}
We say that the action of $G$ on $P$ is \define{proper} if for any proper face $Q$ of $P$ and any element $g$ in the stabilizer  $G_Q$ of $Q$, $g$ fixes $Q$ pointwise (cf. \cite[Section~1]{SteSome}). 
In this case, if $P$ contains the origin then one verifies that $P_g$ is a lattice polytope for all $g$ in $G$,
and $\phi[t]$ is effective by Corollary~\ref{c:fixed}. For an example, we refer the reader to  Corollary~\ref{c:Weyl}. 
%It follows from 
%Corollary~\ref{c:fixed} that if $G$ acts on $P$ properly and $P$ contains a $G$-fixed point, then 
%$\phi[t]$ is effective.  
%whenever $g$ in $G$ fixes a proper face $Q$ of $P$, then $g$ fixes $Q$ point-wise. 
\end{example}

%%%%%%%%%%%%%%%%%%%%%%%%%%%%%%%%%%%%%%%%%%%%%%%
%%%%%%%%%%%%%%%%%%%%%%%%%%%%%%%%%%%%%%%%%%%%%%%
%%%%%%%%%%%%%%%%%%%%%%%%%%%%%%%%%%%%%%%%%%%%%%%
\section{Group actions on cohomology}\label{s:cohomology}
%%%%%%%%%%%%%%%%%%%%%%%%%%%%%%%%%%%%%%%%%%%%%%%
%%%%%%%%%%%%%%%%%%%%%%%%%%%%%%%%%%%%%%%%%%%%%%%
%%%%%%%%%%%%%%%%%%%%%%%%%%%%%%%%%%%%%%%%%%%%%%%

In this section we consider a class of polytopes for which $\phi[t]$ is effective and has a natural geometric interpretation. We continue with the notation of Section~\ref{s2}, Section~\ref{s:Ehrhart} and Section~\ref{s:hstar}.

%We end this section with a geometric interpretation of $\phi[t]$ in certain cases. 
Let $\triangle$ be a smooth, $G$-invariant, $d$-dimensional fan in $M_\R$, and let $X = X(\triangle)$ denote the associated toric variety. Note that $X$ has no odd cohomology, and the action of $G$ on $X$ induces a representation of $G$ on $H^*(X; \C)$. 
If $|\triangle|$ denotes the support of $\triangle$, then let $\psi: |\triangle| \rightarrow \R$ be the piecewise linear function with respect to $\triangle$ that has value $1$ at the primitive lattice points of the rays of $\triangle$. 
Recall that $P$ is \emph{reflexive} if the origin is its unique interior lattice point, and every non-zero lattice point in $M$ lies in the boundary of $mP$  for some positive integer $m$. 

\begin{proposition}\label{p:cohomology}
With the notation above, if $P = \{ v \in M_\R \mid \psi(v) \le 1 \}$ is convex, then $P$ is a $G$-invariant lattice polytope and $\phi_i$ is isomorphic to the $G$-representation $H^{2i}(X; \C)$. In particular, if 
$|\triangle| = M_\R$, then $P$ is reflexive, and
the multiplicities of a fixed irreducible representation in the representations $\phi_i$ form a symmetric, unimodal sequence. 
\end{proposition}
\begin{proof}
We refer the reader to \cite[Section~1]{BriPoincare} and \cite{PayEquivariant} for basic facts on the equivariant cohomology of toric varieties. 
Let $\C[M]^{\triangle}$ denote the \emph{deformed group ring} of $\triangle$. It has a $\C$-vector space basis $\{ x^v \mid v \in |\triangle| \cap M \}$ and multiplication 
\begin{equation*}
x^v \cdot x^w =   \left\{\begin{array}{cl} 
x^{v + w} & \text{if } v,w \textrm{ lie in a common cone in } \triangle%;
 \\ 0 & \text{otherwise}. \end{array}\right. 
\end{equation*}
The action of $G$ on the $T$-fixed points of $X$ induces (via GKM localization) an action of $G$ on the $T$-equivariant cohomology ring $H^*_T X$. 
In particular,  there is a natural $G$-equivariant, graded isomorphism $\C[M]^{\triangle} \cong H^*_T (X; \C)$ (in fact, both sides are naturally isomorphic to the complex Stanley-Reisner ring of $\triangle$). 
Here $\C[M]^{\triangle}$ inherits a natural grading from $\psi: |\triangle| \rightarrow \R$  (see the above discussion), and we consider $H^{2m}_T X$ to have degree $m$. In particular, the representation of $G$ on 
 $H^{2m}_T (X; \C)$ is isomorphic to $\chi_{mP} - \chi_{(m - 1)P}$ for all positive integers $m$. 

There is a natural $G$-invariant graded isomorphism  $H^*_T X \cong H^*X \otimes 
H^*_T (\pt)$, where if $N =  \Hom(M, \Z)$, then $H^*_T (\pt)$ is naturally isomorphic to $\Sym^\bullet N$. 
In particular, we have an isomorphism of graded, $G$-representations $H^*_T (X; \C) \cong H^*(X; \C) \otimes \Sym^\bullet N_\C$. Here, if we fix a basis for $M$ and $g \in G$ acts on $M$ via an integer matrix $A$, then $g$ acts on $N$ via the inverse transpose of $A$. 
If $\{ \lambda_i \}$ denote the eigenvalues of $A$, then the eigenvalues of $A^{-1}$ are the conjugates 
$\{ \bar{\lambda_i} \}$. Since $A$ is integer valued, we conclude that $A$ and the inverse transpose of $A$ have the same eigenvalues and hence we have an isomorphism of $G$-representations $M_\C \cong N_\C$. We conclude that we have the following equality in $R(G)[[t]]$
\[
(1 - t) \sum_{m \ge 0} \chi_{mP} t^m = \sum_{i = 0}^d H^{2i}(X; \C)t^i   \cdot \sum_{m \ge 0} \Sym^m M_\C t^m, 
\]
  and the result follows by Lemma~\ref{l:exterior}. The second statement is well-known and follows from the fact that the Hard Lefschetz theorem holds when $X$ is projective, and the observation that taking the cap product with a hyperplane class commutes with the action of $G$ on the cohomology of $X$ (see, for example, \cite[p. 64]{StaNumber}).  Finally, if $|\triangle| = M_\R$ and $P$ is convex, then $P$ is reflexive by definition. 
  %  for a smooth, projective variety and represent the $G$-action 
\end{proof}

\begin{remark}\label{r:convex}
In the above proposition, the assumption that $P = \{ v \in M_\R \mid \psi(v) \le 1 \}$ is convex is not essential. That is, with the notation above, one can easily extend the definition of $\phi[t]$ to the `star-shaped complex' $P = \{ v \in M_\R \mid \psi(v) \le 1 \}$, and then the proof of Proposition~\ref{p:cohomology} holds verbatim. 
\end{remark}

If $R$ is a reduced, crystallographic root system of rank $d$, then the hyperplanes orthogonal to the roots of $R$ determine a smooth, projective, $d$-dimensional fan $\triangle_R$ in the weight lattice $M$ of $R$, called the \define{Coxeter fan} of $R$. The associated toric variety $X_R = X_R(\triangle_R)$ is the normalization of the closure of a generic torus orbit in the homogeneous space $G/B$ associated to $R$, and the induced action of the Weyl group $W$ on the cohomology 
$H^*(X_R; \C)$ has been studied by 
Procesi~\cite{ProToric}, Stanley \cite[p. 529]{StaLog}, Dolgachev, Lunts~\cite{DLCharacter}, Stembridge~\cite{SteSome, SteEulerian} and Lehrer~\cite{LehRational}. 

\begin{remark}\label{r:Weyl}
Using an intricate case-by-case argument, Stembridge gave a description of the $W$-representation $H^*(X_R; \C)$ 
as an explicit (ungraded) permutation representation \cite{SteSome},
and 
%If $\chi[R; t] = \sum_{i = 0}^d H^{2i}(X_R; \C) t^i$ denotes the corresponding graded representation, then he 
proved that 
$\lef  \sum_{i = 0}^d H^{2i}(X_R; \C) t^i, 1 \rig = (1 + t)^d$ (see the proof of Lemma~3.2 in \cite{SteSome}).  
Lehrer proved that the sign representation of $W$ does not appear in $H^*(X_R; \C)$ \cite[Theorem~3.5 (iii)]{LehRational}, and gave a formula for the multiplicity of the reflection representation $M_\C$ in 
$H^{2i}(X_R; \C)$  \cite[Corollary~4.6]{LehRational}. 
\end{remark}

With the notation above, let $G = W$ %let $M$ be the weight lattice of $R$, 
and consider the action $\rho: W \rightarrow GL(M)$. 
 Let $\psi: M_\R \rightarrow \R$ be the piecewise linear function with respect to $\triangle_R$ that has value $1$ at the primitive lattice points of the rays of $\triangle_R$, and let 
$P = P_R = \{ v \in M_\R \mid \psi(v) \le 1 \}$. The action of $W$ on $P$ is proper in the sense of Example~\ref{e:proper}, and for every $w$ in $W$, $\triangle_R$ restricts to a non-singular fan $\triangle_R^w$ in the subspace $M^w$ of $M_\R$ fixed by $w$
(see Section~2 in \cite{SteSome}). Moreover, $P_w =  \{ v \in M^w \mid \psi^w(v) \le 1 \}$, where 
$\psi^w = \psi|_{M^w}$ is the piecewise  linear function that has value $1$ at the primitive lattice points of the rays of $\triangle_R^w$. %The $h$-polynomial 

We recover the character formula for the graded
$W$-representation $H^*(X_R; \C)$ due to Procesi, Dolgachev and Lunts, and Stembridge \cite[Corollary~1.6]{SteSome}. 

\begin{corollary}\label{c:Weyl}
With the notation above, $\phi[t] =  \sum_{i = 0}^d H^{2i}(X_R; \C) t^i$, and
\[
\phi[t](w) = \frac{ h_{\triangle_R^w}(t) \det(I - \rho(w)t) }{ (1 - t)^{\dim M^w} }. 
\]
\end{corollary}
\begin{proof}
The fact that $\phi[t] =  \sum_{i = 0}^d H^{2i}(X_R; \C) t^i$ follows immediately from Proposition~\ref{p:cohomology} (and Remark~\ref{r:convex}). By Remark~\ref{r:smooth} and Lemma~\ref{l:character},
\[
\frac{ \phi[t](w)}{ (1 - t) \det(I - \rho(w)t) } = \sum_{m \ge 0} \chi_{mP}(w) t^m = \sum_{m \ge 0} f_{P_w}(m) t^m =   
\frac{ h_{\triangle_R^w}(t) }{ (1 - t)^{\dim P_w + 1} }. 
\]
\end{proof}

\begin{remark}\label{r:Dave}
With the notation above, if $R$ is irreducible, then $P = \{ v \in M_\R \mid \psi(v) \le 1 \}$ is convex if $R$ has type $A$, $B$, $C$ or $D$. On the other hand, as explained to the author by Dave Anderson, $P$ is not convex when $R = G_2$. 
\end{remark}

\begin{remark}\label{r:descent}
If $\{ s_1, \ldots, s_d \}$ denotes a set of simple reflections in $W$, then the \emph{length} $l(w)$ of an element $w$ in $W$ is the minimum length of any factorization $w = s_{i_1}\cdots s_{i_r}$, and the 
\emph{descent set} of $w$ is $D(w) = \{ i \mid l(ws_i) < l(w) \}$. It is a theorem of Bj\"orner \cite[Theorem~2.1]{BjoSome} that 
\[
h_{\triangle_R}(t) = \sum_{w \in W} t^{|D(w)|}. 
\]
In particular, if $R = A_{d}$, then $t h_{\triangle_R}(t)$ is the $d^\mth$ Eulerian polynomial (cf. Section~\ref{s:hypercube}).  
\end{remark}

We refer the reader to \cite{SteSome} for explicit computations of the characters $\phi[t]$. 

\begin{example}\label{e:typeA}
If $R = A_{d}$, then $W = \Sym_{d + 1}$ acts on $M = \Z^{d + 1}/\Z(1, \ldots, 1)$ by permuting coordinates, and $P$ is the reflexive polytope with vertices given by the images of $\{ e_{i_1} + \cdots + e_{i_l} \mid i_1 < \cdots < i_l, 1 \le l \le d - 1 \}$ in $\Z^{d + 1}$. We consider this example in more detail in Section~\ref{s:hypercube} (cf. Remark~\ref{r:Karu}). 
\end{example}

\begin{example}\label{e:typeB}
If $R = B_d$ (or $R = C_d$), then $W$ is the group of signed permutations, and the image of 
$W \rightarrow GL_d(\Z)$ consists  of all matrices with entries in $\{0 , \pm 1 \}$, and  precisely one non-zero element in each row and column. In this case, $W$ may be interpreted as the full symmetry group of the hypercube $P = [-1, 1]^d$.

\end{example}

%%%%%%%%%%%%%%%%%%%%%%%%%%%%%%%%%%%%%%%%%%%%%%%
%%%%%%%%%%%%%%%%%%%%%%%%%%%%%%%%%%%%%%%%%%%%%%%
%%%%%%%%%%%%%%%%%%%%%%%%%%%%%%%%%%%%%%%%%%%%%%%
%%%%%%%%%%%%%%%%%%%%%%%%%%%%%%%%%%%%%%%%%%%%%%%
\section{Symmetries of the hypercube}\label{s:hypercube}
%%%%%%%%%%%%%%%%%%%%%%%%%%%%%%%%%%%%%%%%%%%%%%%
%%%%%%%%%%%%%%%%%%%%%%%%%%%%%%%%%%%%%%%%%%%%%%%
%%%%%%%%%%%%%%%%%%%%%%%%%%%%%%%%%%%%%%%%%%%%%%%
%%%%%%%%%%%%%%%%%%%%%%%%%%%%%%%%%%%%%%%%%%%%%%%

The goal of this section is to explicitly describe the equivariant Ehrhart theory of the $d$-dimensional hypercube. We continue with the notation of Section~\ref{s2}, Section~\ref{s:Ehrhart} and Section~\ref{s:hstar}. 

Let $M = \Z^d$ %with standard basis $e_1, \ldots, e_d$, 
and let $P = [0,1]^d$. %\{ v = (v_1, \ldots, v_d) \in M_\R \mid 0 \le v_1 \le 1 \}$.  
The Ehrhart polynomial and $h^*$-polynomial of $P$ are well-known. That is, $f_P(m) = (m + 1)^d$ and $h^*(t) = A(d; t) =  \sum_{i = 0}^{d - 1} A(d,i) t^i$, where $A(d,i)$ is the number of  permutations of $d$ elements with $i$ descents (see, for example, \cite[p. 28]{BRComputing}). 
The polynomial $tA(d,t)$ is  the $d^\mth$ \define{Eulerian polynomial} and its coefficients are called \define{Eulerian numbers}. 
% i.e. $\{ i \mid w(i) > w(i + 1) \}$. 

The full symmetry group of $P$ is the Coxeter group of type $B_d$ %of order $2^d d!$,  
consisting of all \emph{signed permutations} of $d$ elements (cf. Example~\ref{e:typeB}). More precisely, $B_d$ acts faithfully on $\C^d$ and its image in 
$\GL(d)$ consists of all matrices with entries in $\{0 , \pm 1 \}$, and  precisely one non-zero element in each row and column. In particular, we may view $\Sym_d$ as a subgroup of $B_d$. With the notation of Section~\ref{s2}, observe that $B_d$ preserves the lattice $M$ and leaves $P$ invariant `up to translation'. Note that the diagonal matrix $(-1,1, \ldots, 1)$ does not fix any lattice points, and hence Lemma~\ref{l:bad} implies that $\phi[t]$ is not a polynomial. We demonstrate this failure explicitly below for  the %$2$-cube. 
square. 

\begin{example}\label{e:breakdown}
Let $M = \Z^2$ and let $P$ be the convex hull of $(0,0)$, $(1,0)$, $(0,1)$ and $(1,1)$. 
The group $G = B_2 = \left\langle \sigma, \tau \mid \sigma^4 = \tau^2 = 1,  \sigma \tau = \tau \sigma^3\right\rangle$ may be viewed as the subgroup of $\Sym_4$ generated by $\sigma = (1234)$ and $\tau = (12)(34)$. It has $5$ conjugacy classes $\{1 \}$, $\{ \sigma, \sigma^3 \}$, $\{ \sigma^2 \}$, $\{ \tau, \tau \sigma^2 \}$ and 
$\{ \tau \sigma, \tau \sigma^3 \}$.  The group $G$ acts linearly on $M$ via 
\begin{displaymath}
\sigma \mapsto \left(\begin{array}{c c}
0 & -1 \\ 
1 & 0 
\end{array}\right), \: \:
\tau \mapsto \left(\begin{array}{c c}
-1 & 0 \\ 
0 & 1 
\end{array}\right),
\end{displaymath}
with corresponding $1$-dimensional representation $\det(\rho)$.
If $\epsilon$ denotes the restriction of the sign representation of $\Sym_4$ to $G$, then 
%the irreducible representations of $G$ are $\{ 1, \epsilon, \det(\rho), \epsilon \otimes \det(\rho), M_\C \}$. 
%One computes that if $r$ is a positive, even integer, then $\phi_{rP}(t)$ is given by  
%\[
%\phi_1 =  \frac{r(r + 6)}{8} +  \frac{r(r + 2)}{8}\epsilon +  \frac{r(r - 2)}{8}\det(\rho) +  \frac{r(r + 2)}{8}\epsilon \otimes \det(\rho) + \frac{r^2 + 2r - 4}{4}M_\C,
%\]
%\[
%\phi_2 =  \frac{r(r + 2)}{8} +  \frac{r(r - 2)}{8}\epsilon +  \frac{(r - 2)(r - 4)}{8}\det(\rho) +  \frac{r(r - 2)}{8}\epsilon \otimes \det(\rho) + \frac{r(r - 2)}{4}M_\C. 
%\]
%If $r$ is a positive, odd integer, then $\phi_{rP}(t)$ is given by  
%\[
%\phi_1 =  \frac{r^2 + 4r - 5}{8} +  \frac{r^2 - 1}{8}\epsilon +  \frac{r^2 - 1}{8}\det(\rho) +  \frac{r^2 + 4r + 3}{8}\epsilon \otimes \det(\rho) + \frac{r^2 + 2r - 3}{4}M_\C,
%\]
%\[
%\phi_2 =  \frac{r^2 + 4r + 3}{8} +  \frac{r^2 - 1}{8}\epsilon +  \frac{r^2 - 8r + 7}{8}\det(\rho) +  \frac{r^2 - 4r - 5}{8}\epsilon \otimes \det(\rho) + \frac{(r - 1)^2}{4}M_\C, 
%\]
%and 
%\[
%\phi[t] - 1 - \phi_1 t - \phi_2 t^2 = (\epsilon \otimes \det(\rho) - 1) \frac{t^3}{1 + t}.
%\]
\[
\phi[t] = 1 + \epsilon \cdot \det(\rho) t + (1 - \epsilon \cdot  \det(\rho)) \frac{t^2}{1 + t}. 
\]
The subgroup $\Sym_2 \subseteq B_2$ consists of the identity element and $\tau \sigma = (13)$. Observe that if we restrict to $\Sym_2$,  then $\phi[t] = 1 + t$.  
\end{example}

\begin{remark}
If one considers the corresponding action of $B_d$ on $2P$, then $\phi[t]$ is effective and has an explicit geometric interpretation (Example~\ref{e:typeB}). 
\end{remark}

For the remainder of the section, we consider the action of the symmetric group $G = \Sym_d \subseteq B_d$. In this case, $P$ is invariant under the action of $G$, and %$M_\C$ 
$\rho: G \rightarrow GL(M)$ is the standard representation of the symmetric group \eqref{e:standard}. 
%and we will see that $\phi[t]$ is a polynomial.  
If $g \in G$ has cycle type 
$(\mu_1, \ldots, \mu_r)$, then $P_g$ is isomorphic to an $r$-dimensional cube, and hence 
  \[
\sum_{m \ge 0} f_{P_g}(m)t^m = \frac{A(r; t)}{(1 - t)^{r + 1}} =  \frac{A(r; t)\prod_i (1 + t + \cdots +  t^{\mu_i - 1})}{(1 - t)\prod_i (1 - t^{\mu_i})}. 
\]
We conclude that %the evaluation $\phi[t](g)$ of the character of $\phi[t]$ at $g$ is given by
\begin{equation}\label{e:charcube}
\phi[t](g)  = A(r; t)\prod_i (1 + t + \cdots +  t^{\mu_i - 1}). 
\end{equation}
Recall from Section~\ref{s:cohomology} that the Coxeter fan $\triangle_{A_{d - 1}}$ is the fan determined by the hyperplanes associated with the root system $A_{d - 1}$, and $\Sym_d$ acts on the cohomology 
$H^*(X_{A_{d - 1}}; \C)$ of the associated $(d - 1)$-dimensional toric variety $X_{A_{d - 1}} = X_{A_{d - 1}}(\triangle_{A_{d - 1}})$. 

%The \define{Coxeter complex} $\Sigma$ of type $A_{d - 1}$ is the fan determined to the hyperplanes associated with the root system $A_{d - 1}$. The associated $(d - 1)$-dimensional toric variety $X = X(\Sigma)$ is smooth and projective, and the action of $\Sym_d$ on $\Sigma$ induces a graded representation of $\Sym_d$ on the cohomology $H^*(X; \C)$ of $X$. These representations, and similar representations for other Weyl groups, have been studied by
%Procesi~\cite{ProToric}, Stanley \cite[p. 529]{StaLog}, Dolgachev, Lunts~\cite{DLCharacter}, Stembridge~\cite{SteSome, SteEulerian} and Lehrer~\cite{LehRational}. We note that the Hard Lefschetz theorem implies that the multiplicities of a fixed irreducible representation in the representations $H^{2i}(X; \C)$ form a symmetric,  unimodal sequence (cf. Proposition~\ref{p:cohomology}). 
% 

\begin{lemma}\label{l:typeA}
With the notation above, the representation $\phi_i$ is isomorphic to the representation of $\Sym_d$ on
  $H^{2i}(X_{A_{d - 1}}; \C)$. 
  In particular, 
the multiplicities of a fixed irreducible representation in the representations $\phi_i$ form a symmetric, unimodal sequence. 
\end{lemma}
\begin{proof}
This follows by comparing  \eqref{e:charcube} with Corollary~\ref{c:Weyl}, using the fact that if
$g \in G$ has cycle type 
$(\mu_1, \ldots, \mu_r)$, then $\dim M^g = r$ and $\det(I - \rho(g)t) = \prod_i (1 - t^{\mu_i})$. 
%It is a result of Stembridge \cite[Proposition~3.3]{SteEulerian} that if $g \in G$ has cycle type 
%$(\mu_1, \ldots, \mu_r)$, then the character of $\sum_{i = 0}^d H^{2i}(X; \C)t^i$ evaluated at $g$ is equal to the right hand side of \eqref{e:charcube} ($H^*(X; \C)$ has no odd cohomology). 
\end{proof}
 
 \begin{remark}\label{r:Karu}
As explained to the author by Kalle Karu, Lemma~\ref{l:typeA} has the following geometric explanation. 
The $d!$  smooth cones 
\[
\{   \lef e_{w(1)}, e_{w(1)} + e_{w(2)}, \cdots, e_{w(1)} + e_{w(2)} + \cdots + e_{w(d)} \rig \mid w \in \Sym_d \} \textrm{ in } \R^d
\]
form a smooth fan $\Sigma$ supported on $(\R_{\ge 0})^d$, and Proposition~\ref{p:cohomology} implies that $\phi_i$  is isomorphic to the representation $H^{2i}(X(\Sigma); \C)$ on the cohomology of the associated toric variety $X(\Sigma)$. Moreover, the projection $\R^d \rightarrow \R^d/\R(1, \ldots, 1)$ is a morphism of fans from $\Sigma$ to $\triangle_{A_{d - 1}}$ (see Example~\ref{e:typeA}), and induces an $\Sym_d$-equivariant isomorphism $H^* (X(\Sigma)) \cong H^* (X_{A_{d - 1}})$. 
\end{remark}
 
 \begin{remark}
The fact that $\phi[t] = t^{d - 1}\phi[t^{-1}]$ also follows the observation that $2P$ is a reflexive polytope and Corollary~\ref{c:reflexive}.
\end{remark}
 
 In fact, Stembridge proves that the representation $H^*(X_{A_{d - 1}}; \C)$ is an explicit graded permutation representation  \cite{SteEulerian}. We briefly recall his decomposition of $H^*(X_{A_{d - 1}}; \C)$ into isotypic components, and refer the reader to  \cite[Section~4]{SteEulerian} for more details.  Recall from Example~\ref{e:symmetric} that 
 the irreducible representations $\chi^{\lambda}$ of $\Sym_d$ 
%and their corresponding characters $\chi_\lambda$ 
are indexed by partitions $\lambda$ of $d$. If $D_\lambda$ %= \{ (i,j) \mid 1 \le i \le n, 1 \le j \le \lambda_i\}$ 
 denotes the Young diagram of $\lambda$, then a \emph{tableau} $T$ with shape $\lambda$ is a function $T: D_\lambda \rightarrow \N$ such that $T$ is weakly increasing along rows and strictly increasing down columns. If $m_j(T)$ equals the number of times $j$ appears in $T$, then
 we say that $T$ is \emph{admissible} if 
   $S^+(T) := \{ j \in \Z_{> 0} \mid m_j(T) > 0 \} = \{ 1, 2, \ldots, k \}$ for some $k \in \N$. 
 A \define{marked tableau} is a pair $(T,f)$, where $T$ is an admissible tableau and 
 $f: S^+(T) \rightarrow \N$ satisfies $1 \le f(k) < m_j(T)$. The \define{index}  of a marked tableau is 
$\ind(T,f) = \sum_{j \in S^+} f(j)$, and if $\lambda = (d)$, then the pair $(T, \emptyset)$, where $T: D_\lambda \rightarrow \N$ is the zero function, is a marked tableau of index zero.  
Observe that if $m_j(T) = 1$ for some $j \in S^+(T)$, then there are no marked tableaux $(T,f)$. 
For example,
the marked tableaux corresponding to partitions of  $2$, with indices $0$ and $1$ respectively, are
\[
\left( \: \begin{tabular}{  | c |  c |}
\hline
$0$ & $0$ \\
\hline
\end{tabular}\,, \:
f = \emptyset \right), 
\left( \: \begin{tabular}{ | c |  c |}
\hline
$1$ & $1$ \\
\hline
\end{tabular}\,, \:
f(1) = 1 \right), 
%\left( \: 
%\begin{tabular}{ | c |  c |}
%\hline
%$1$  \\
%\hline
%$1$   \\
%\hline
%\end{tabular}\,, \:
%f(1) = 1 \right), 
\]
and the marked tableaux corresponding to partitions of  $3$, with indices $0$, $1$, $1$, $2$ and $1$ respectively, are
\[
\left( \: \begin{tabular}{ | c |  c | c |}
\hline
$0$ & $0$  & $0$ \\
\hline
\end{tabular}\,, \:
f = \emptyset \right), 
\left( \: \begin{tabular}{| c |  c | c |}
\hline
$0$ & $1$ & $1$ \\
\hline
\end{tabular}\,, \:
f(1) = 1 \right), 
\left( \: \begin{tabular}{| c |  c | c |}
\hline
$1$ & $1$ & $1$ \\
\hline
\end{tabular}\,, \:
f(1) = 1 \right), 
\]
\[ 
\left( \: \begin{tabular}{| c |  c | c |}
\hline
$1$ & $1$ & $1$ \\
\hline
\end{tabular}\,, \:
f(1) = 2 \right), 
\left( \: 
\begin{tabular}{ | c |  c | c |}
\hline
$0$ & $1$ \\
\hline
$1$   \\ \cline{1-1}
%$1$ & $1$ \\
%\hline
\end{tabular}\,, \:
f(1) = 1 \right). 
\]
Let $P_\lambda(t) = \sum_{i = 0}^{d - 1} p_{i,\lambda} t^i$, where $p_{i,\lambda}$ denotes the multiplicity of $\chi^\lambda$ in $H^{2i}(X_{A_{d - 1}}; \C)$. 

\begin{theorem}\cite[Theorem~4.2]{SteEulerian}
With the notation above, for any partition $\lambda$ of $d$, 
\[
P_\lambda(t) = \sum_{(T,f)} t^{\ind(T,f)},
\]
where $(T,f)$ is summed  over all marked tableaux of shape $\lambda$. 
\end{theorem}
Stembridge also used the above theorem to give a combinatorial proof that the coefficients of $P_\lambda(t)$ are symmetric and unimodal. For $d \le 2r$, observe that there are no marked tableaux of shape 
$\lambda = (d - r,1^{r})$, and hence the corresponding irreducible representations 
$\chi^{(d - r,1^{r})} =   \bigwedge^{r} \chi^{(d -  1, 1)}$ do not appear in $H^*(X_{A_{d - 1}}; \C)$. In particular, the sign representation $\chi^{(1^d)}$ does not appear for $d \ge 2$ (cf. Remark~\ref{r:Weyl}). 
It is a result of Stembridge that $P_{(d)}(t) = (1 + t)^d$ \cite[Lemma~3.2]{SteSome}, and a result of Lehrer that
$P_{(d - 1, 1)}(t)= (d - 2)t(1 + t)^{d - 3}$ \cite[Theorem~4.5]{LehRational}.

We summarize the results of this discussion in the following proposition.

\begin{proposition}\label{p:cube}
If $G = \Sym_d$ acts on the $M = \Z^d$ via the standard representation and $P = [0,1]^d$, then 
$\phi_i$ is isomorphic to the representation of $\Sym_d$ on $H^{2i}(X_{A_{d - 1}}; \C)$, where $X_{A_{d - 1}}$ is the $(d - 1)$-dimensional smooth, projective toric variety associated to the Coxeter complex of $A_{d - 1}$. Moreover, $\phi_i$ is a permutation representation, and if we write 
\[
\phi[t] = \sum_{|\lambda| = d} P_\lambda (t) \chi^{\lambda}, 
\]
where $\{ \chi^\lambda \mid |\lambda| = d \}$ are the irreducible representations of $\Sym_d$, 
then the coefficients of $P_\lambda (t)  = t^{d - 1}P_\lambda (t^{-1})$ are unimodal, and 
\[
P_\lambda(t) = \sum_{(T,f)} t^{\ind(T,f)},
\]
where $(T,f)$ is summed  over all marked tableaux of shape $\lambda$. In particular, 
 $P_{(d)}(t) = (1 + t)^d$ and
$P_{(d - 1, 1)}(t)= (d - 2)t(1 + t)^{d - 3}$. 
\end{proposition}

For example, $\phi[t] = 1 + t$ when $d = 2$, $\phi[t] = 1 + (2 + \chi^{(2,1)})t + t^2$ when $d = 3$, 
and  $\phi[t] = 1 + (3 + 2\chi^{(3,1)} + \chi^{(2,2)})t + (3 + 2\chi^{(3,1)} +  \chi^{(2,2)})t^2 + t^3$ when $d = 4$.

\begin{remark}\label{r:refinement}
If $\lambda$ is a partition of $n$, then 
an (admissible) tableau $T$ of shape $\lambda$ is \emph{semi-standard} if $S^+(T) = \{ 1, \ldots, n \}$, 
and it is well known that $\dim \chi^{\lambda}$ equals the number of semi-standard tableaux of shape $\lambda$. In particular, by considering dimensions of representations in Proposition~\ref{p:cube}, one obtains Stembridge's refinement of the Eulerian numbers \cite{SteEulerian}. 
\end{remark}

%\begin{remark}
%One can also deduce that $\phi_i$ is an effective representation from Corollary~\ref{c:fixed}. 
%\end{remark}

%%%%%%%%%%%%%%%%%%%%%%%%%%%%%%%%%%%%%%%%%%%%%%%
%%%%%%%%%%%%%%%%%%%%%%%%%%%%%%%%%%%%%%%%%%%%%%%
%%%%%%%%%%%%%%%%%%%%%%%%%%%%%%%%%%%%%%%%%%%%%%%
%%%%%%%%%%%%%%%%%%%%%%%%%%%%%%%%%%%%%%%%%%%%%%%
\section{Applications to rational polytopes}\label{s:pip}
%%%%%%%%%%%%%%%%%%%%%%%%%%%%%%%%%%%%%%%%%%%%%%%
%%%%%%%%%%%%%%%%%%%%%%%%%%%%%%%%%%%%%%%%%%%%%%%
%%%%%%%%%%%%%%%%%%%%%%%%%%%%%%%%%%%%%%%%%%%%%%%
%%%%%%%%%%%%%%%%%%%%%%%%%%%%%%%%%%%%%%%%%%%%%%%

In this section, we present an example to demonstrate how one can use the theory developed in the previous sections to explicitly describe the Ehrhart theory of certain classes of rational polytopes. 
 We continue with the notation of Section~\ref{s2}, Section~\ref{s:Ehrhart} and Section~\ref{s:hstar}. 

Recall that 
%These representations are interesting from a combinatorial perspective because they
the characters $\{ \chi_{mP} \}_{m \ge 0}$ 
 encode the Ehrhart theory of the rational polytopes $P_g = \{ u \in P \mid g \cdot u = u \}$ for all $g$ in $G$. 
More specifically, %the value of the character of $\chi_{mP}$ at $g \in G$ is given by 
by Lemma~\ref{l:character}, 
$\chi_{mP}(g) = f_{P_g}(m)$, where $f_{P_g}(m)$ is the Ehrhart quasi-polynomial of $P_g$. In the case when $P$ is a simplex, we have deduced an explicit formula for the generating series of  $f_{P_g}(m)$
(Proposition~\ref{p:simplex}).

A rational polytope is a \define{pseudo-integral polytope} or PIP if its Ehrhart quasi-polynomial is a polynomial. The first examples of  rational PIP's  in arbitrary dimension were found by De Loera and McAllister 
in \cite{DLMVertices}. This work was extended later by McAllister and Woods in \cite{MWMinimum}, who found PIP's with arbitrary denominator (see Section~\ref{s:basicEhrhart})  in all dimensions. Below we use our techniques to construct a new family of PIP's.
% in order to illustrate our techniques.  

Let $G = \Sym_{2n}$ and consider the natural action of $G$ on $M = \Z^{2n}/\Z(1,\ldots, 1)$. That is, 
$M_\C = \chi^{(2n - 1, 1)}$ is the quotient of the standard representation \eqref{e:standard}  by a $1$-dimensional invariant subspace. Let $P$ be the \emph{standard reflexive simplex} with vertices 
given by the images of the standard basis vectors $e_1, \ldots, e_{2n}$ in $\Z^{2n}$. 
%In what follows, with a slight abuse of notation, we fix the basis $e_1, \ldots, e_{2n - 1}$ for $M$ so that 
%$e_{2n} = - e_1 - \ldots - e_{2n - 1}$. 
It is well known that 
$h^*(t) = 1 + t + \cdots + t^{2n - 1}$, and hence Proposition~\ref{p:simplex} implies that 
$\phi[t] = 1 + t + \cdots + t^{2n - 1}$. By Proposition~\ref{p:simplex}, if $g \in G$ has cycle type 
$(\mu_1, \ldots, \mu_r)$, then $P_g$ is isomorphic to the $(r - 1)$-dimensional simplex with vertices given by the images of 
\[
\frac{e_1 + \ldots + e_{\mu_1}}{\mu_1}  , \frac{e_{\mu_1 + 1} + \ldots + e_{\mu_1 + \mu_2}}{\mu_2}, \ldots, 
 \frac{e_{\mu_1 + \cdots + \mu_{r - 1} + 1}  + \cdots + e_{\mu_1 + \cdots + \mu_r}}{\mu_r},
\]
and
  \[
\sum_{m \ge 0} f_{P_g}(m)t^m  =  \frac{1 + t + \cdots + t^{2n - 1}}{\prod_i (1 - t^{\mu_i})}. 
\]
For example, $g = (12 \ldots n)$ has cycle type $(n, 1^n)$, and hence $P_g$ is the $n$-dimensional rational polytope with denominator $n$ and  vertices given by the images of 
\[
\frac{e_1 + \ldots + e_{n}}{n}, e_{n + 1}, \ldots, e_{2n}.
\] 
Moreover, 
 \[
\sum_{m \ge 0} f_{P_g}(m)t^m  =  \frac{1 + t + \cdots + t^{2n - 1}}{(1 - t^n)(1 - t)^n} = \frac{1 + t^n}{(1 - t)^{n + 1}}. 
\]
In particular, $P_g$ is a PIP with Ehrhart quasi-polynomial $f_{P_g}(m) = \binom{m + n}{n} + \binom{m}{n}$. For $n \ge 2$, note that the coefficient of $t$ in the numerator $1 + t^n$ is strictly less than the coefficient of $t^n$, and hence the inequality \eqref{e:easyinequality} implies that $f_{P_g}(m)$ is not the Ehrhart polynomial of a $n$-dimensional lattice polytope.

%%%%%%%%%%%%%%%%%%%%%%%%%%%%%%%%%%%%%%%%%%%%%%%
%%%%%%%%%%%%%%%%%%%%%%%%%%%%%%%%%%%%%%%%%%%%%%%
%%%%%%%%%%%%%%%%%%%%%%%%%%%%%%%%%%%%%%%%%%%%%%%
%%%%%%%%%%%%%%%%%%%%%%%%%%%%%%%%%%%%%%%%%%%%%%%
\section{Centrally symmetric polytopes}\label{s:central}
%%%%%%%%%%%%%%%%%%%%%%%%%%%%%%%%%%%%%%%%%%%%%%%
%%%%%%%%%%%%%%%%%%%%%%%%%%%%%%%%%%%%%%%%%%%%%%%
%%%%%%%%%%%%%%%%%%%%%%%%%%%%%%%%%%%%%%%%%%%%%%%
%%%%%%%%%%%%%%%%%%%%%%%%%%%%%%%%%%%%%%%%%%%%%%%

The goal of this section is to explicitly describe the equivariant Ehrhart theory of centrally symmetric polytopes. We continue with the notation of Section~\ref{s2}, Section~\ref{s:Ehrhart} and Section~\ref{s:hstar}.

Let $P$ be a $d$-dimensional lattice polytope in a lattice $M$ of rank $d$. 
Let $G = \Z/2\Z$ with generator $\sigma$, and let $\chi: G \rightarrow \C$ be the linear character sending $\sigma$ to $-1$, so that the irreducible representations of $G$ are $\{ 1, \chi \}$.
The polytope $P$ is \define{centrally symmetric} if it is $G$-invariant with respect to the 
 action of $G$ on $M$ in which $\sigma$ acts via the diagonal matrix $A = (-1, \ldots, -1)$. 

First observe that the origin is an interior lattice point of $P$, and hence Corollary~\ref{c:fixed} implies that $\phi[t]$ is effective. %the representations $\phi_i$ are effective. 
On the other hand, we claim that the polynomial $\phi[t]$ is determined by the $h^*$-polynomial of $P$. Indeed,  observe that  %$P_\sigma = \{ 0 \}$, and 
%the origin is an interior lattice point of $P$, and 
$P_\sigma = \{ 0 \}$, and
%In this case, the representations $H^*_i$ are determined by the Ehrhart $h^*$-polynomial $h^*(t)$. 
%More specifically, observe that the origin is an interior lattice point of $P$, and $P_\sigma = \{ 0 \}$. 
hence, by Lemma~\ref{l:character},  
\[
\sum_{m \ge 0} \chi_{mP}(\sigma) t^m = \sum_{m \ge 0} f_{P_g}(m) t^m = \frac{1}{1 - t} = \frac{(1 + t)^d}{(1 - t)\det(I - tA)}. 
\]
It easily follows that %$\phi_i = 0$ for $i >  d$, and  
\begin{equation}\label{e:centrally}
\phi_i = \frac{h_i^* + \binom{d}{i}}{2} + \frac{h_i^* - \binom{d}{i}}{2}\chi. 
\end{equation}
 %for $0 \le i \le d$. 
 In particular, observe that the effectiveness of the representations $\phi_i$ is equivalent to the lower bound $h_i^* \ge \binom{d}{i}$.  %for $1 \le i \le d$. 
 The latter lower bound was proved by Bey, Henk and Wills in \cite[Remark 1.6]{BHWNotes}. As they observe, 
the lower bound may be deduced from results of Betke and McMullen \cite[Theorem~2]{BMLattice} and 
 Stanley \cite{StaNumber} (cf.  \cite[Remark 1.6]{BHWNotes}), and equality is achieved  when $P$ is the $d$-dimensional cross-polytope (see the discussion below). 
 Hence we deduce an alternative proof of this lower bound. 
 In conclusion, we have established the following result. 
 
 %the following corollary using a result. 
 
 \begin{corollary}\label{c:cs}
 With the notation above, if $G = \Z/2\Z$ and $P$ is a centrally symmetric polytope, then
 $\phi[t]$ is effective, and 
  %the virtual representations $\phi_i$ are effective representations.   Moreover, 
  the multiplicity of the trivial representation in $\phi_i$ is at least $\binom{d}{i}$.
  \end{corollary}

%In particular, observe that the effectiveness of the representations $H^*_i$ (proved in Corollary~\ref{c:vanishing}) is equivalent to the lower bound $h_i^* \ge \binom{d}{i}$ for $1 \le i \le d$. The latter lower bound was proved by Bey, Henk and Wills in \cite[Remark 1.6]{BHWNotes}. As they observe, 
%the lower bound may be deduced from results of Betke and McMullen \cite[Theorem~2]{BMLattice} and 
 %Stanley \cite{StaNumber} (cf.  \cite[Remark 1.6]{BHWNotes}). 

\excise{
follows from a result of Betke and McMullen \cite[Theorem~2]{BMLattice} that the $h$-vector $h_\tau(t)$ of any regular, lattice triangulation $\tau$ of $P$ satisfies $h_\tau(t) \le h^*(t)$, together with the corresponding lower bound for $h$-vectors of centrally symmetric polytopes, that was conjectured by Bj\"orner and proved by Stanley in \cite{StaNumber}. 
An alternative proof of the lower bound for $h^*$-vectors is given by Bey, Henk and Wills in \cite[Remark 1.6]{BHWNotes}. \comment{Stanley assumes simplicial polytope though? Does he need to? What about Bey, Henk, Wills?}
 %was conjectured follows from a result of 
 }
 
 Similarly, 
 the representations $\chi_{mP}$ are determined by the Ehrhart polynomial $f_P(m)$. Indeed, 
 %the quasi-polynomial
  $L(m) = \chi_{mP}$ is the polynomial 
 \[
L(m) = \frac{f_P(m) +  1}{2}  + \frac{f_P(m) -  1}{2} \chi.
 \]

 In order to compute some explicit examples, we recall the classification of non-singular, centrally symmetric, reflexive polytopes by Klyachko and Voskresenski{\u\i} in \cite{KVToroidal}. We note that these results have been extended by Ewald \cite{EwaClassification}, Casagrande \cite{CasCentrally} and Nill \cite{NilClassification}. Recall %from Corollary~\ref{c:reflexive} 
 that a $d$-dimensional lattice polytope $P$ in $M_\R$ is \emph{reflexive} if the origin is the unique interior lattice point of $P$ and every non-zero lattice point in $M$ lies in the boundary of $mP$ for some positive integer $m$. We say that $P$ is \define{non-singular} if, furthermore, the vertices of each facet of $P$ form a basis for $M$. 
 
 Let $\{ e_1, \ldots, e_d \}$ be a lattice basis for $M$, and  let  $V(2k,d)$ be the 
  lattice polytope with vertices $\{ \pm e_i, \; \pm (e_1 + \cdots + e_{2k}) \mid 1 \le i \le d \}$ for $0 \le 2k \le d$. The polytope $V(0,d)$ is called the $d$-dimensional \define{cross-polytope} and, if $d$ is even, then 
$V(d,d)$ is called the \define{ Klyachko-Voskresenski{\u\i} polytope} or KV-polytope.
 
 \begin{theorem}\cite{KVToroidal}
% Let $M$ be a rank $d$ lattice with basis $\{ e_1, \ldots, e_d \}$. Then 
 The $d$-dimensional, non-singular, reflexive, centrally symmetric lattice polytopes are (up to unimodular transformation) precisely the polytopes  $V(2k,d)$  for $0 \le 2k \le d$.
 % where 
 %$V(2k, d)$ is the lattice polytope with vertices $\{ \pm e_i, \; \pm (e_1 + \cdots + e_{2k}) \mid 1 \le i \le d \}$. 
 \end{theorem}
 
 By Proposition~\ref{p:cohomology}, if $P$ is a non-singular, reflexive centrally symmetric polytope, then $\phi[t]$ is a polynomial describing the representations of $\Z/2\Z$ on the graded pieces of the  cohomology of the toric variety associated to the fan over the faces of $P$. In particular, the multiplicities of a fixed irreducible representation in the representations $\phi_i$ form a symmetric, unimodal sequence.  By \eqref{e:centrally}, the unimodality of 
 these sequences is equivalent to the inequalities 
 \[
 h_i^* - h_{i - 1}^* \ge \binom{d}{i} - \binom{d}{i - 1},
 \]
 for $0 \le i \le \frac{d}{2}$, which may also be deduced from \cite{StaNumber}. 
 
%  Note that under the assumptions of the above theorem, the polynomials $\phi_i$ may be regarded as the representations on the graded pieces of the  cohomology of the toric variety associated to the fan over the faces of $P$ (Lemma~\ref{l:cohomology}).
  %\comment{need to discuss unimodality} 
   
   Our next goal is to compute the representations explicitly in this case. 
%Our next goal is to compute the polynomials $\phi[t]$ associated to these polytopes. 
Since $V(2k, d)$ may be regarded as the free sum  of $V(2k, 2k)$ and $d - 2k$ copies of the interval $[-1,1]$, it follows from Proposition~\ref{p:Braun} that 
\[
\phi_{V(2k, d)}[t] = \phi_{V(2k, 2k)}[t](1 + t)^{d - 2k}. 
\]
By the above discussion, it remains to compute the $h^*$-polynomial of a KV-polytope.
% By the above discussion, it will be enough to compute the $h^*$-polynomials of the polytopes $V(2k, d)$. 
%Since $V(2k, d)$ may be regarded as the \emph{free sum} (see, for example, \cite{BraEhrhart}) of $V(2k, 2k)$ and $d - 2k$ copies of the interval $[-1,1]$, it follows from a result of Braun \cite[Theorem~1]{BraEhrhart} that 
%\[
%h^*_{V(2k, d)}(t) = h^*_{V(2k, 2k)}(t)(1 + t)^{d - 2k}.  
%\]
%Thus, in order to compute the representations $H^*_i$ for non-singular, reflexive, centrally symmetric polytopes, we are left with computing 
%Hence it remains to compute the $h^*$-polynomial of a KV-polytope. 
If $\triangle$ denotes the fan over the faces of $V(2k, 2k)$, then Remark~\ref{r:smooth} implies that
\[
h^*_{V(2k, 2k)}(t) = h_{\triangle}(t) = \sum_{i = 0}^{2k} f_{i - 1} t^i (1 - t)^{2k - i},
\]
where $f_{i - 1}$ equals the number of faces of $V(2k, 2k)$ of dimension $i - 1$, and $f_{-1} = 1$. 
%It follows from a result of Betke and McMullen \cite{BMLattice} that the coefficients of the $h^*$-polynomial of a non-singular, reflexive polytope coincides with its $h$-vector. That is, 
%\[
%h^*_{V(2k, 2k)}(t) = \sum_{i = 0}^{2k} f_{i - 1} t^i (1 - t)^{2k - i},
%\]
%where $f_{i - 1}$ equals the number of faces of $P$ of dimension $i - 1$, and $f_{-1} = 1$. 
%, since $V(2k, 2k)$ is smooth and reflexive, $h^*_{V(2k, 2k)}(t)$
It follows from Section~3 in \cite{CasCentrally} that 
%\[
%f_{i - 1} = \binom{2k + 1}{i}[\binom{i}{i - k} + \cdots + \binom{i}{i + k}]. 
%\]
%In particular, 
$f_{i - 1} = \binom{2k + 1}{i}2^i$ for $0 \le i \le k$, and $\sum_i h^*_i = f_{d - 1} = (2k + 1)\binom{2k}{k}$. 
For example, $h^*_{V(2, 2)}(t) =  t^2 + 4t + 1$ and $h^*_{V(4, 4)}(t) =  t^4 + 6t^3 + 16t^2 + 6t + 1$. 

For any $d$, let $T_d(t) = \sum_{i = 0}^d T(d,i) t^i$ be the symmetric polynomial of degree $d$ with 
$T(d, i)$ equal to the coefficent of $t^i$ in $\sum_{j = 0}^d \binom{d + 1}{j}2^j  t^j (1 - t)^{d - j}$ for $0 \le i \le \frac{d}{2}$. In particular, if $d = 2k$, then $T_d(t) = h^*_{V(2k, 2k)}(t)$ and $T_d(1) = (2k + 1)\binom{2k}{k}$. Here we set $\binom{2k}{k} = 1$ if $k = 0$. The following lemma shows that the coefficients $T(d,i)$ may be interpreted as partial sums of rows of Pascal's triangle. 

%polynomials $T_d(t)$ may be defined recursively. 

\begin{lemma}\label{l:hyperplane}
With the notation above, if  $i < \frac{d}{2}$, then
\[
T(d, i) = T(d - 1, i) + T(d - 1, i - 1), 
\]
and if $d = 2k$, then 
\[
T(2k, k) = 2T(2k - 1, k - 1) + \binom{2k}{k}.
\] 
In particular, $T(d,i) = \sum_{j = 0}^{i} \binom{d + 1}{j}$ for $0 \le i \le \frac{d}{2}$. 
\end{lemma}
\begin{proof}
For $0 \le i < \frac{d}{2}$, $T(d,i)$ is equal to the coefficent of $t^i$ in
\[
\sum_{j = 0}^d ( \binom{d}{j} +  \binom{d}{j - 1})2^j  t^j (1 - t)^{d - j}
\]
\[
= (1 - t) \sum_{j = 0}^{d - 1} \binom{d}{j} 2^j  t^j (1 - t)^{d - 1 - j}                   +  2t \sum_{j = 0}^{d - 1}  \binom{d}{j} 2^j  t^{j} (1 - t)^{d - 1 - j}
\]
\[
= (1 + t) \sum_{j = 0}^{d - 1} \binom{d}{j} 2^j  t^j (1 - t)^{d - 1 - j}.                 
%  +  t \sum_{j = 0}^{d - 1}  \binom{d}{j} 2^j  t^{j} (1 - t)^{d - 1 - j}.
\]
Hence $T(d, i) = T(d - 1, i) + T(d - 1, i - 1)$, as desired. It follows that if $d$ is odd, then $T_d(1) = 2T_{d - 1}(
1)$, and if $d = 2k$ is even, then $T_{2k}(1) = 2T_{2k - 1}(1) - 2T(2k - 1, k - 1) + T(2k, k)$. 
On the other hand, 
%1) - 2 T(d - 1, \frac{ d - 1}{2})$
by the above discussion,  if $d = 2k$, then $T_{2k}(1) = (2k + 1)\binom{2k}{k}$,  %The result follows by 
%Equating the last two expressions and rearranging gives 
% and we compute
%\[
%(d + 1)\binom{d}{\frac{d}{2}} = T(d; 1) =  
%\]
and we deduce the equality
\[
T(2k, k) = 2T(2k - 1, k - 1) + (2k + 1)\binom{2k}{k} - 4(2k - 1)\binom{2k - 2}{k - 1}.
\] 
One verifies that the latter two terms sum to $\binom{2k}{k}$. The final statement now follows by a simple induction argument. 
\end{proof}

%We deduce the following interpretation of the polynomials $T_d(t)$. 
%We immediately deduce from the above lemma that 
Using the above lemma and induction, we deduce the following expressions for the $h^*$-polynomials of the KV-polytopes 
\[
h^*_{V(2k, 2k)}(t) = \sum_{i = 0}^{k} \binom{2k + 1}{i} (t^i + \cdots + t^{2k - i}) = \sum_{i = 0}^{k} \binom{2i}{i} t^i (1 + t)^{2(k - i)}. 
\]

We summarize the results of this discussion in the following proposition.

\begin{proposition}\label{p:centrally}
 The $d$-dimensional, non-singular, reflexive, centrally symmetric lattice polytopes are (up to unimodular transformation) precisely the polytopes  $V(2k,d)$  for $0 \le 2k \le d$, with $h^*$-polynomials 
 \[
 h^*_{V(2k,d)}(t) = \sum_{i = 0}^{k} \binom{2i}{i} (1 + t)^{d - 2i},
 \]
 and %$\phi[t]$ defined by \eqref{e:centrally}.
\[
\phi[t] =  \frac{h^*_{V(2k,d)}(t)  + (1 + t)^d}{2} +  \frac{h^*_{V(2k,d)}(t)  - (1 + t)^d}{2}\chi. 
\]

\end{proposition}

\section{Open questions and conjectures}\label{s:open}

We end the paper by presenting some open problems and directions for future research. 
 We continue with the notation of Section~\ref{s2}, Section~\ref{s:Ehrhart} and Section~\ref{s:hstar}.

Recall that the power series $\phi[t]$ is \emph{effective} if all the virtual representations $\phi_i$ are effective representations. The main open problem is to characterize when $\phi[t]$ is effective. 
We have seen that if the toric variety $Y$ and ample line bundle $L$ associated to the $G$-invariant lattice polytope $P$ admit a $G$-invariant non-degenerate hypersurface, then $\phi[t]$ is effective, and, in particular, $\phi[t]$ is a polynomial (Theorem~\ref{t:nondeg}). We offer the following conjecture. 

% there exists a $G$-invariant non-degenerate hypersurface with Newton polytope $P$, then $\phi[t]$ is effective, and, in particular, $\phi[t]$ is a polynomial (Theorem~\ref{t:nondeg}). We offer the following conjecture. 

%The main open problem is to characterize when $\phi[t]$ is a polynomial and when $\phi[t]$ is effective. 
%We do not offer any conjectural answer, but rather we assert that these two problems are the same. 

\begin{conjecture}\label{c:big}
With the notation above, the following conditions are equivalent
\begin{itemize}
\item $(Y,L)$ admits a $G$-invariant non-degenerate hypersurface, % with Newton polytope $P$
\item $\phi[t]$ is effective,
\item  $\phi[t]$ is a polynomial. 
\end{itemize}

\end{conjecture}

Observe that the equivalence of the last two conditions in the above conjecture holds in dimension $2$ by Corollary~\ref{c:basic}. 

We have seen that if $P$ is a simplex (Proposition~\ref{p:simplex}), or if $G = \Sym_d$ and $P = [0,1]^d$ (Proposition~\ref{p:cube}), then the representations $\phi_i$ are permutation representations.
 When $\phi[t]$ is effective, the ungraded character
$\phi[1] = \sum_i \phi_i$ has non-negative integer values (Proposition~\ref{p:rational}), and if $G$ is cyclic of prime order, then this guarantees that $\phi[1]$ is a permutation representation. Moreover, Stembridge proved that 
the $\phi[1]$ is a permutation representation in the example of Corollary~\ref{c:Weyl} (Remark~\ref{r:Weyl}).  
 On the other hand, $\phi_1$ need not be a permutation representation, even when $\phi[t]$ is effective (Example~\ref{e:counter}). 

\begin{conjecture}
If $\phi[t]$ is effective, then $\phi[1]$ is a permutation representation. 
\end{conjecture}

If $\phi[t]$ is not a polynomial, then one can still define a rational character $\phi[1]$ (Lemma~\ref{l:rational}). We conjecture that $\phi[1]$ in fact takes (non-negative) integer  values (cf. Proposition~\ref{p:rational}). 

\begin{conjecture}\label{c:integer}
%The rational character $\phi[1]$ is integer-valued. That is, 
For any $g$ in $G$, 
\[
\phi[1](g)
=   \frac{ \dim(P_g)! \vol(P_g) \det(I - \rho(g))_{(M^g)^\perp }}{ \ind(P_g) }
\]
is a non-negative integer. 
\end{conjecture}

%It would be very useful to have %even a partial %a combinatorial proof of the latter fact, or, more generally, 
%an answer to the following question. 

%\begin{question}
%Do there exist conditions which guarantee that the representations $\phi_i$ are permutation representations? 
%\end{question}

We also offer the following conjecture on the appearance of the trivial representation. 
%In particular, it holds whenever $\phi_i$ is a permutation representation (Lemma~\ref{l:permutation}). 

\begin{conjecture}
If $\phi[t]$ is a polynomial and $h^*_i > 0$, then the trivial representation occurs with non-zero multiplicity in $\phi_i$. 
\end{conjecture}

It may also be interesting to consider the above conjecture in the special case when $P$ contains a $G$-fixed lattice point in its interior (in this case, $h_i > 0$ for $0 \le i \le d$). 
The conjecture holds in all examples presented in this paper. 
%We note that the above conjecture 
In particular, it holds when $P$ is a simplex by Proposition~\ref{p:simplex}, and when $\dim P = 2$ by Corollary~\ref{c:reciprocity} and Corollary~\ref{c:basic}.

Lastly, we list a number of possible directions for future research. 

\begin{itemize}

\item Resolve the conjectures and questions above when $G = \Z/2\Z$. 

\item  %It may be interesting to 
Consider the lower degree coefficients    $L_i(m)$ in Section~\ref{s:Ehrhart}. Can one say something about their minimal periods, even in the case when $P$ is a simplex?

\item What can one say about the asymptotic behavior of $\phi_{mP}[t]$ for $m$ sufficiently large and divisible (cf. Corollary~\ref{c:multiple}, \cite{BYoHilbert})? 

\item What can be deduced from the equivariant Riemann-Roch formula \cite[Appendix~I]{GGKMoment}? This was suggested to the author by Roberto Paoletti. 

\item Can one develop a natural equivariant version of weighted Ehrhart theory \cite{YoWeightI}?

\item Describe the equivariant Ehrhart theory of the permutahedron (cf. \cite{PosPermutohedra}). 

\item Compute the representations $\phi[t]$ when $G$ is the full symmetry group of the polytopes appearing in Proposition~\ref{p:centrally}.

\end{itemize}

\excise{
\begin{itemize}
\item What can one say about the lower terms $L_i(m)$ in Section~\ref{s:Ehrhart}? What are the minimal periods? Even in the case when $P$ is a simplex.

\item Conjecture that if $P$ contains a ($G$-fixed ?) point in interior, then $\phi[t]$ is effective and the trivial representation occurs with non-zero multiplicity in $\phi_i$ for $0 \le i \le d$. 

\item If $\phi[t]$ is a polynomial, does the trivial representation always  occur with non-zero multiplicity in $\phi_1$? Need to go to dimension $3$ (degree $2$) to get a counter-example. When $G = 1$, can certainly have $h_1 = 0$. What if $h_1 \ne 0$? What about the other coefficients? Is it true that $h^*_i > 0$ implies that the trivial representation occurs? Not true for $h^*_1$ without polynomial assumption. 

\item Conjecture that the following are equivalent: linear system base point free, non-degenerate invariant hypersurface, $\phi[t]$ effective, $\phi[t]$ polynomial. One direction is done.

\end{itemize}
}

\bibliographystyle{amsplain}
\bibliography{alan}

\end{document}